\documentclass[12pt]{article}

\usepackage{latexsym,amsmath,amscd,amssymb,graphics,float}
\usepackage{enumerate}

\usepackage{graphicx,lscape}

\usepackage[colorlinks]{hyperref}
\usepackage{url}

\begin{document}

\newenvironment{proof}[1][Proof]{\textbf{#1.} }{\ \rule{0.5em}{0.5em}}

\newtheorem{theorem}{Theorem}[section]
\newtheorem{definition}[theorem]{Definition}
\newtheorem{lemma}[theorem]{Lemma}
\newtheorem{remark}[theorem]{Remark}
\newtheorem{proposition}[theorem]{Proposition}
\newtheorem{corollary}[theorem]{Corollary}
\newtheorem{example}[theorem]{Example}

\numberwithin{equation}{section}
\newcommand{\ep}{\varepsilon}
\newcommand{\R}{{\mathbb  R}}
\newcommand\C{{\mathbb  C}}
\newcommand\Q{{\mathbb Q}}
\newcommand\Z{{\mathbb Z}}
\newcommand{\N}{{\mathbb N}}

\newcommand{\bfi}{\bfseries\itshape}

\newsavebox{\savepar}
\newenvironment{boxit}{\begin{lrbox}{\savepar}
\begin{minipage}[b]{15.5cm}}{\end{minipage}\end{lrbox}
\fbox{\usebox{\savepar}}}

\title{{\bf Bifurcations of limit cycles of perturbed completely integrable systems}}
\author{R\u{a}zvan M. Tudoran and Anania G\^irban}

\date{}
\maketitle \makeatother

\begin{abstract}
The main purpose of this article is to study from the geometric point of view the problem of limit cycles bifurcation of perturbed completely integrable systems. 
\end{abstract}

\medskip

\textbf{AMS 2010}: 34C07; 37G15; 37J20; 37J40.


\section{Introduction}
\label{section:one}
The main purpose of this article is to study from the geometric point of view, the problem of limit cycles bifurcation of perturbed real analytic completely integrable systems. By geometric approach of the problem we mean the use of those perturbations which preserves some of the geometry of the completely integrable system. This problem was also analyzed in the geometric context for a particular three dimensional completely integrable system, namely Euler's equations of the free rigid body, \cite{bermejo1}. Also a non-geometric version of the problem, using averaging methods, was recently analyzed in \cite{bermejo2}.

Our approach to the bifurcation problem is a local one, and makes use of the Darboux normal form of completely integrable systems. In order to do that, in the second section of the article we give a geometric local picture of a general real analytic completely integrable system. More precisely, we prove that the Hamilton-Poisson realizations (on Poisson manifolds with bracket coming from a Nambu bracket) of general real analytic completely integrable systems, are rigid under arbitrary analytic change of coordinates. As a consequence, we obtain a unitary approach of two local normal forms of a general completely integrable system, i.e., the Darboux normal form, and a new local normal form, which provides a Flow Box Theorem in the category of completely integrable systems. Returning to the limit cycles bifurcation problem, our approach proves to be tight related with the problem of limit cycles bifurcations of planar analytic Hamiltonian dynamical systems, since by the Darboux normal form, the (local) nontrivial dynamics of an $n-$dimensional analytic completely integrable system, is actually described by that of a planar analytic Hamiltonian dynamical system. Hence, we shall use the results in the planar theory, in order to analyze (locally) similar bifurcation phenomena for $n-$dimensional analytic completely integrable systems. The purpose of the third section of this article is to recall some of the main approaches to the problem of limit cycles bifurcations, for planar analytic Hamiltonian dynamical systems. The fourth section is the main part of the article. Here we analyze the problem of limit cycles bifurcation of geometrically perturbed real analytic completely integrable systems, from two points of view. More precisely, in the first part of this section one considers the problem of limit cycles bifurcations, in the context of a general perturbation which preserves an a-priori fixed symplectic leaf of the Poisson configuration manifold of the unperturbed completely integrable system. On the other hand, in the second part of this section we deal with the same problem, this time analyzing the class of perturbations that preserves all of the regular symplectic leaves of the Poisson configuration manifold of the unperturbed completely integrable system. The fifth section is concerned with the illustration of the main results of this article, to a concrete dynamical system. The example we shall consider for this purpose is a real analytic completely integrable system which generates the so called \textit{Jacobi hyperelliptic functions} \cite{bates}.

\section{A geometric picture of completely integrable systems, and two local normal forms}

The aim of this section is to present a geometric local picture of a general real analytic completely integrable system. More precisely, we prove that the Hamilton-Poisson realizations (on Poisson manifolds with bracket coming from a Nambu bracket) of general real analytic completely integrable systems, are rigid under arbitrary analytic change of coordinates. As a consequence, we obtain a unitary approach of two local normal forms of a general completely integrable system, i.e., the Darboux normal form, and a new local normal form, which provides a Flow Box Theorem in the category of completely integrable systems. Note that all the results in this section will also hold in the smooth category; the only reason we work in the analytic category is to give a unitary presentation of the paper.

Let us start by recalling from \cite{tudoran} the Hamiltonian realization procedure of a completely integrable system. For similar Hamilton-Poisson and respectively Nambu-Poisson formulations of completely integrable systems, see e.g., \cite{g1}, \cite{g2}, \cite{nambu1}, \cite{nambu2}, \cite{ratiurazvan}.

Recall that a real analytic \textit{completely integrable system} is a $\mathcal{C}^{\omega}$ (i.e., real analytic) differential system defined on a domain (i.e., connected open subset) $\Omega\subseteq\R^n$,
\begin{equation}\label{sys}
\left\{\begin{array}{l}
\dot x_{1}=X_1(x_1,\dots,x_n)\\
\dot x_{2}=X_2(x_1,\dots,x_n)\\
\cdots\\
\dot x_{n}=X_n(x_1,\dots,x_n),\\
\end{array}\right.
\end{equation}
(where $X_1,X_2,\dots,X_n\in\mathcal{C}^{\omega}(\Omega,\R)$ are real analytic functions), which admits a set of real analytic first integrals, $C_1,\dots,C_{n-2},C_{n-1}:\Omega\rightarrow \R$, functionally independent almost everywhere with respect to the $n-$dimensional Lebesgue measure.

Since the real analytic functions $C_1,\dots,C_{n-2},C_{n-1}:\Omega\rightarrow \R$ are constants of motion of the analytic vector field $X=X_1\dfrac{\partial}{\partial{x_1}}+\dots+X_n \dfrac{\partial}{\partial{x_n}}\in\mathfrak{X}(\Omega)$, we get that for each $i\in\{1,\dots,n-1\}$,
$$
0=\left(\mathcal{L}_{X}C_i\right)(\overline{x})=\langle\nabla C_i(\overline{x}),X(\overline{x})\rangle,
$$
for every $\overline{x}=(x_1,\dots,x_n)\in\Omega$ (where $\mathcal{L}_{X}$ stands for the \textit{Lie derivative along the vector field} $X$, $\langle \cdot,\cdot\rangle$ is the canonical inner product on $\R^n$, and $\nabla$ stands for the gradient with respect to $\langle \cdot,\cdot\rangle$).

Hence, as shown in \cite{tudoran}, the $\mathcal{C}^{\omega}$ vector field $X$ is proportional with the vector field $\star (\nabla C_1\wedge\dots\wedge \nabla C_{n-1})$, where $\star$ stands for the Hodge star operator for multi-vector fields. It may happen that the domain of analyticity of the proportionality rescaling function, is a proper subset of $\Omega$. In order to simplify the notations, we shall work in the sequel on this set (which is supposed to be a domain), denoted also by $\Omega$.  

Consequently, the analytic vector field $X$ can be realized on the domain $\Omega\subseteq \R^n$ as a Hamilton-Poisson vector field $X_H \in\mathfrak{X}(\Omega)$, with respect to the Hamiltonian function $H:=C_{n-1}$, and the Poisson bracket of class $\mathcal{C}^{\omega}$ given by
$$
\{f,g\}_{\nu;C_1,\dots,C_{n-2}}\cdot \mathrm{d}x_1\wedge\dots\wedge \mathrm{d}x_n:=\nu \cdot \mathrm{d}C_1\wedge\dots \mathrm{d}C_{n-2}\wedge \mathrm{d}f\wedge \mathrm{d}g,
$$
for every $f,g\in\mathcal{C}^{\omega}(\Omega,\mathbb{R})$, where $\nu\in\mathcal{C}^{\omega}(\Omega,\R)$ is the proportionality rescaling function. 

Note that the real analytic functions $C_1,\dots,C_{n-2}$, form a complete set of Casimir functions of the Poisson bracket $\{\cdot,\cdot\}_{\nu;C_1,\dots,C_{n-2}}$.

The above defined Poisson bracket can also be obtained from the rescaled canonical Nambu bracket on $\Omega\subseteq \R^n$ as follows:
$$
\{f,g\}_{\nu;C_1,\dots,C_{n-2}}=\nu \cdot \{C_1,\dots,C_{n-2},f,g\}:=\nu \cdot \dfrac{\partial(C_1,\dots,C_{n-2},f,g)}{\partial(x_1,\dots,x_n)}.
$$

Consequently, the Hamiltonian vector field $X=X_H\in\mathfrak{X}(\Omega)$, is acting on an arbitrary analytic real function $f\in \mathcal{C}^{\omega}(\Omega,\R)$, as: 
$$X_{H}(f)=\{f,H\}_{\nu;C_1,\dots,C_{n-2}}\in \mathcal{C}^{\omega}(\Omega,\R).$$

Hence, the differential system \eqref{sys}, can be locally realized on the Poisson manifold $\left(\Omega,\{\cdot,\cdot\}_{\nu;C_1,\dots,C_{n-2}}\right)$ as a real analytic \textit{Hamiltonian dynamical system} of the type
\begin{equation}\label{systy}
\left\{\begin{array}{l}
\dot x_{1}=\{x_1,H\}_{\nu;C_1,\dots,C_{n-2}}\\
\dot x_{2}=\{x_2,H\}_{\nu;C_1,\dots,C_{n-2}}\\
\cdots \\
\dot x_{n}=\{x_n,H\}_{\nu;C_1,\dots,C_{n-2}}.\\  
\end{array}\right.
\end{equation}

Consequently, the restriction to $\Omega$ of the components $X_i$, of the vector field $X$ which generates the system \eqref{sys}, admits the following formulation:
$$X_i=\nu \cdot \dfrac{\partial(C_1,\dots,C_{n-2},x_i,H)}{\partial(x_1,\dots,x_n)},$$
for every $i\in\{1,\dots,n\}$.

Since the main purpose of this section is a local picture of general completely integrable systems, we shall analyze in the sequel the behavior of real analytic completely integrable systems, relative to a local change of coordinates. More precisely, we will prove that Hamilton-Poisson realizations of completely integrable systems, modeled on Poisson manifolds with bracket coming from a Nambu bracket, \textit{are rigid under arbitrary real analytic change of coordinates}. 

In order to do that, let us fix an analytic change of coordinates between two domains, $\Omega, W \subseteq\mathbb{R}^{n}$:
$$\overline{x}=(x_1,\dots,x_n)\in\Omega\subseteq\mathbb{R}^{n} \mapsto \overline{y}=(y_1,\dots,y_n)=\Phi(\overline{x})\in W=\Phi(\Omega)\subseteq\mathbb{R}^{n},$$ 
i.e., $\Phi=(\Phi_1 ,\dots, \Phi_n):\Omega\rightarrow W$ is real analytic, invertible, and the inverse is also real analytic (equivalently, $\Phi$ is said to be an \textit{analytic diffeomorphism}).

\begin{proposition}\label{lacc}
Let $\Phi:\Omega\rightarrow W$ be an analytic diffeomorphism between two domains $\Omega, W \subseteq\mathbb{R}^{n}$, and let $\nu,C_1,\dots,C_{n-2}\in\mathcal{C}^{\omega}(\Omega,\mathbb{R})$ be some fixed real analytic functions. Then $\Phi$ is a Poisson isomorphism between the Poisson manifold $\left(\Omega,\{\cdot,\cdot\}_{\nu;C_1,\dots,C_{n-2}}\right)$ and $\left(W,\{\cdot,\cdot\}_{\nu_{\Phi};\Phi_{\star}C_1,\dots,\Phi_{\star}C_{n-2}}\right)$, where $\nu_{\Phi}=:\Phi_{\star}\nu \cdot \Phi_{\star}\operatorname{Jac}(\Phi)$, and $\Phi_{\star}F:=F\circ\Phi^{-1}$, for every $F\in\mathcal{C}^{\omega}(\Omega,\mathbb{R})$.
\end{proposition}
\begin{proof}
The conclusion is equivalent to $\Phi^{-1}:W\rightarrow \Omega$ is a Poisson isomorphism, i.e.,
\begin{equation}\label{psdif}
\{f,g\}_{\nu;C_1,\dots,C_{n-2}}\circ\Phi^{-1}=\{f\circ\Phi^{-1},g\circ\Phi^{-1}\}_{\nu_{\Phi};\Phi_{\star}C_1,\dots,\Phi_{\star}C_{n-2}},
\end{equation}
for every $f,g\in\mathcal{C}^{\omega}(\Omega,\mathbb{R})$, where 
$$
\nu_{\Phi}=(\nu\circ\Phi^{-1})\cdot\left(\dfrac{\partial(\Phi_1,\dots,\Phi_n)}{\partial(x_1,\dots,x_n)}\circ\Phi^{-1}\right),
$$
and $\Phi=(\Phi_1 ,\dots, \Phi_n):\Omega\rightarrow W$.

In order to prove the relation \eqref{psdif}, let $f,g\in\mathcal{C}^{\omega}(\Omega,\mathbb{R})$ and $\overline{y}\in W$, be given. Applying the chain rule, we get the following equalities:
\begin{align*}
&(\{f,g\}_{\nu;C_1,\dots,C_{n-2}}\circ\Phi^{-1})(\overline{y})=\{f,g\}_{\nu;C_1,\dots,C_{n-2}}(\Phi^{-1}(\overline{y}))=\nu(\Phi^{-1}(\overline{y}))\cdot\dfrac{\partial(C_1,\dots,C_{n-2},f,g)}{\partial(x_1,\dots,x_n)}(\Phi^{-1}(\overline{y}))\\
&=\nu(\Phi^{-1}(\overline{y}))\cdot\dfrac{\partial((C_1\circ\Phi^{-1})\circ\Phi,\dots,(C_{n-2} \circ\Phi^{-1})\circ\Phi,(f\circ\Phi^{-1})\circ\Phi,(g\circ\Phi^{-1})\circ\Phi)}{\partial(x_1,\dots,x_n)}(\Phi^{-1}(\overline{y}))\\
&=\nu(\Phi^{-1}(\overline{y}))\cdot\dfrac{\partial(C_1 \circ\Phi^{-1},\dots,C_{n-2}\circ\Phi^{-1},f\circ\Phi^{-1},g\circ\Phi^{-1})}{\partial(y_1,\dots,y_n)}(\overline{y})\cdot\dfrac{\partial(\Phi_1,\dots,\Phi_n)}{\partial(x_1,\dots,x_n)}(\Phi^{-1}(\overline{y}))\\
&=\left[(\nu\circ\Phi^{-1})(\overline{y})\cdot\left(\dfrac{\partial(\Phi_1,\dots,\Phi_n)}{\partial(x_1,\dots,x_n)}\circ\Phi^{-1}\right)(\overline{y})\right]\cdot\dfrac{\partial(C_1 \circ\Phi^{-1},\dots,C_{n-2}\circ\Phi^{-1},f\circ\Phi^{-1},g\circ\Phi^{-1})}{\partial(y_1,\dots,y_n)}(\overline{y})\\
&=\nu_{\Phi}(\overline{y})\cdot \dfrac{\partial(\Phi_{\star}C_1,\dots,\Phi_{\star}C_{n-2},\Phi_{\star}f,\Phi_{\star}g)}{\partial(y_1,\dots,y_n)}(\overline{y})=\{\Phi_{\star}f,\Phi_{\star}g\}_{\nu_{\Phi};\Phi_{\star}C_1,\dots,\Phi_{\star}C_{n-2}}(\overline{y})\\
&=\{f\circ\Phi^{-1},g\circ\Phi^{-1}\}_{\nu_{\Phi};\Phi_{\star}C_1,\dots,\Phi_{\star}C_{n-2}}(\overline{y}),
\end{align*}
which imply the conclusion.
\end{proof}

From Proposition \eqref{lacc} we obtain that for each real analytic function $H\in\mathcal{C}^{\omega}(\Omega,\mathbb{R})$, the push forward by $\Phi$ of the Hamiltonian vector field $X_H \in\mathfrak{X}(\Omega)$, defined on the Poisson manifold $\left(\Omega,\{\cdot,\cdot\}_{\nu;C_1,\dots,C_{n-2}}\right)$, is a Hamiltonian vector filed modeled on the Poisson manifold $\left(W,\{\cdot,\cdot\}_{\nu_{\Phi};\Phi_{\star}C_1,\dots,\Phi_{\star}C_{n-2}}\right)$, given by $\Phi_{\star}X_{H}=X_{\Phi_{\star}H}$.

Consequently, from the Hamiltonian realization \eqref{systy} of the completely integrable system \eqref{sys}, we obtain the expression of the push forward through an arbitrary analytic change of coordinates, $\Phi:\Omega\rightarrow W$, of the vector field $X=X_1 \dfrac{\partial}{\partial{x_1}}+\dots+X_n \dfrac{\partial}{\partial{x_n}}$, associated to the completely integrable system \eqref{sys}. More precisely, we have shown the following result.

\begin{theorem}\label{tsv}
Let $$X = \sum_{i=1}^{n}\nu \cdot \dfrac{\partial(C_1,\dots,C_{n-2},x_i,H)}{\partial(x_1,\dots,x_n)}\cdot\dfrac{\partial}{\partial{x_i}},$$ be the vector field associated to the real analytic completely integrable system \eqref{sys}, written as a Hamiltonian dynamical system of the form \eqref{systy}, with Hamiltonian $H:=C_{n-1}$, defined on the Poisson manifold $\left(\Omega,\{\cdot,\cdot\}_{\nu;C_1,\dots,C_{n-2}}\right)$. Let  $\Phi:\Omega\rightarrow W=\Phi(\Omega)$ be an analytic diffeomorphism. 

Then $\Phi_{\star}X$, the push forward by $\Phi$ of the vector field $X$, is a Hamiltonian vector field, with Hamiltonian $\Phi_{\star}H=\Phi_{\star}C_{n-1}$, defined on the Poisson manifold $\left(W,\{\cdot,\cdot\}_{\nu_{\Phi};\Phi_{\star}C_1,\dots,\Phi_{\star}C_{n-2}}\right)$, and has the expression
\begin{equation}\label{vpush}
\Phi_{\star}X = \sum_{i=1}^{n}\nu_{\Phi} \cdot \dfrac{\partial(\Phi_{\star}C_1,\dots,\Phi_{\star}C_{n-2},y_i,\Phi_{\star}H)}{\partial(y_1,\dots,y_n)}\cdot\dfrac{\partial}{\partial{y_i}},
\end{equation}
where $\nu_{\Phi}=\Phi_{\star}\nu\cdot \Phi_{\star}\operatorname{Jac}(\Phi)$, and $(y_1,\dots,y_n)=\Phi(x_1,\dots,x_n)$, denote the local coordinates on the domain $W$.
\end{theorem}

The above theorem provides a useful method to construct (local) normal forms of a real analytic completely integrable system. Note that Theorem \eqref{tsv} remains valid also for weaker classes of differentiability, the reason we state it in the analytic category regards only its applications in the next sections of this article. 

Let us give now two (local) normal forms of a real analytic completely integrable system, the first one, also called the Darboux normal form, being crucial for the main purpose of this article. For another approach of the Darboux normal form see, e.g., \cite{berme1}.  

\begin{theorem}\label{darbnf}
Let $$X = \sum_{i=1}^{n}\nu \cdot \dfrac{\partial(C_1,\dots,C_{n-2},x_i,H)}{\partial(x_1,\dots,x_n)}\cdot\dfrac{\partial}{\partial{x_i}},$$ be the vector field associated to the real analytic completely integrable system \eqref{sys}, written as a Hamiltonian dynamical system of the form \eqref{systy}, with Hamiltonian $H:=C_{n-1}$, defined on the Poisson manifold $\left(\Omega,\{\cdot,\cdot\}_{\nu;C_1,\dots,C_{n-2}}\right)$. Let $(\Omega^{\prime},\Phi_1,\Phi_2)$ be a triple consisting of an open subset  $\Omega^{\prime}\subseteq \Omega$, and two analytic functions $\Phi_1,\Phi_2 \in\mathcal{C}^{\omega}(\Omega^{\prime},\mathbb{R})$, such that the map $\Phi:\Omega^{\prime}\rightarrow W^{\prime}=\Phi(\Omega^{\prime})$, given by $\Phi=(\Phi_1,\Phi_2,C_1,\dots,C_{n-2})$, is an analytic diffeomorphism. 

Then $\Phi_{\star}X$, the push forward of the vector field $X$ by $\Phi$, is a Hamiltonian vector field, with Hamiltonian $\Phi_{\star}H=\Phi_{\star}C_{n-1}$, defined on the Poisson manifold $\left(W^{\prime},\{\cdot,\cdot\}_{\nu_{\Phi};\Phi_{\star}C_1,\dots,\Phi_{\star}C_{n-2}}\right)$, and has the expression
\begin{equation}\label{darb}
\Phi_{\star}X = \nu_{\Phi}\cdot\left[ \dfrac{\partial(\Phi_{\star}H)}{\partial y_2}\cdot\dfrac{\partial}{\partial y_1}-\dfrac{\partial(\Phi_{\star}H)}{\partial y_1}\cdot\dfrac{\partial}{\partial y_2}\right],
\end{equation}
where $\nu_{\Phi}=\Phi_{\star}\nu\cdot \Phi_{\star}\operatorname{Jac}(\Phi)$, and $(y_1,\dots,y_n)=\Phi(x_1,\dots,x_n)$, denote the local coordinates on the domain $W^{\prime}$.
\end{theorem}
\begin{proof}
From Theorem \eqref{tsv} we get that the push forward of the vector field $X$ by $\Phi$, is a Hamiltonian vector field with Hamiltonian $\Phi_{\star}H=\Phi_{\star}C_{n-1}$, defined on the Poisson manifold $\left(W^{\prime},\{\cdot,\cdot\}_{\nu_{\Phi};\Phi_{\star}C_1,\dots,\Phi_{\star}C_{n-2}}\right)$, and given by:
\begin{equation}\label{erf}
\Phi_{\star}X = \sum_{i=1}^{n}\nu_{\Phi} \cdot \dfrac{\partial(\Phi_{\star}C_1,\dots,\Phi_{\star}C_{n-2},y_i,\Phi_{\star}H)}{\partial(y_1,\dots,y_n)}\cdot\dfrac{\partial}{\partial{y_i}}.
\end{equation}
Since $\Phi=(\Phi_1,\Phi_2,C_1,\dots,C_{n-2})$, one obtain that for any $(y_1,\dots,y_n)\in W^{\prime}$ we have $\Phi_{\star}C_1(y_1,\dots,y_n)=y_3,\dots,\Phi_{\star}C_{n-2}(y_1,\dots,y_n)=y_n$, and consequently the expression \eqref{erf} becomes:
\begin{align*}
\Phi_{\star}X &= \sum_{i=1}^{n}\nu_{\Phi} \cdot \dfrac{\partial(\Phi_{\star}C_1,\dots,\Phi_{\star}C_{n-2},y_i,\Phi_{\star}H)}{\partial(y_1,\dots,y_n)}\cdot\dfrac{\partial}{\partial{y_i}}=\sum_{i=1}^{n}\nu_{\Phi} \cdot \dfrac{\partial(y_3,\dots,y_n,y_i,\Phi_{\star}H)}{\partial(y_1,\dots,y_n)}\cdot\dfrac{\partial}{\partial{y_i}}\\
&=\nu_{\Phi} \cdot \left[\dfrac{\partial(y_3,\dots,y_n,y_1,\Phi_{\star}H)}{\partial(y_1,\dots,y_n)}\cdot\dfrac{\partial}{\partial{y_1}}+ \dfrac{\partial(y_3,\dots,y_n,y_2,\Phi_{\star}H)}{\partial(y_1,\dots,y_n)}\cdot\dfrac{\partial}{\partial{y_2}}\right]\\
&=\nu_{\Phi}\cdot\left[ \dfrac{\partial(\Phi_{\star}H)}{\partial y_2}\cdot\dfrac{\partial}{\partial y_1}-\dfrac{\partial(\Phi_{\star}H)}{\partial y_1}\cdot\dfrac{\partial}{\partial y_2}\right].
\end{align*}
\end{proof}

Next local normal form provides a \textit{Flow Box Theorem} for the class of completely integrable systems.

\begin{theorem}\label{flowbox}
Let $$X = \sum_{i=1}^{n}\nu \cdot \dfrac{\partial(C_1,\dots,C_{n-2},x_i,H)}{\partial(x_1,\dots,x_n)}\cdot\dfrac{\partial}{\partial{x_i}},$$ be the vector field associated to the real analytic completely integrable system \eqref{sys}, written as a Hamiltonian dynamical system of the form \eqref{systy}, with Hamiltonian $H:=C_{n-1}$, defined on the Poisson manifold $\left(\Omega,\{\cdot,\cdot\}_{\nu;C_1,\dots,C_{n-2}}\right)$. Let $(\Omega^{\prime\prime},\Phi_1)$ be a pair consisting of an open subset  $\Omega^{\prime\prime}\subseteq \Omega$, and an analytic function $\Phi_1 \in\mathcal{C}^{\omega}(\Omega^{\prime\prime},\mathbb{R})$, such that the map $\Phi:\Omega^{\prime\prime}\rightarrow W^{\prime\prime}=\Phi(\Omega^{\prime\prime})$, given by $\Phi=(\Phi_1,H,C_1,\dots,C_{n-2})$, is an analytic diffeomorphism. 

Then $\Phi_{\star}X$, the push forward of the vector field $X$ by $\Phi$, is a Hamiltonian vector field, with Hamiltonian $\Phi_{\star}H=\Phi_{\star}C_{n-1}$, defined on the Poisson manifold $\left(W^{\prime\prime},\{\cdot,\cdot\}_{\nu_{\Phi};\Phi_{\star}C_1,\dots,\Phi_{\star}C_{n-2}}\right)$, and has the expression
\begin{equation}\label{fbox}
\Phi_{\star}X = \nu_{\Phi}\cdot\dfrac{\partial}{\partial y_1},
\end{equation}
where $\nu_{\Phi}=\Phi_{\star}\nu\cdot \Phi_{\star}\operatorname{Jac}(\Phi)$, and $(y_1,\dots,y_n)=\Phi(x_1,\dots,x_n)$, denote the local coordinates on the domain $W^{\prime\prime}$.
\end{theorem}
\begin{proof}
The proof follows mimetically that of Theorem \eqref{darbnf}, the only difference being the definition of $\Phi:=(\Phi_1,H,C_1,\dots,C_{n-2})$, which implies that $$\Phi_{\star}H(y_1,\dots,y_n)=y_2,\Phi_{\star}C_1(y_1,\dots,y_n)=y_3,\dots,\Phi_{\star}C_{n-2}(y_1,\dots,y_n)=y_n,$$ for any $(y_1,\dots,y_n)\in W^{\prime\prime}$.
\end{proof}

\begin{remark}
In order to obtain a "clean" normal form of the dynamical systems generated by the local normal forms \eqref{darb} and \eqref{fbox}  of the vector field $X$, one need to perform the time reparameterization $\mathrm{d}s=\nu_{\Phi}(\overline{y})\mathrm{d}t$, where $s$ denotes the new time coordinate (provided that $\nu_{\Phi}$ is nonvanishing).  
\end{remark}

\section{Limit cycles bifurcations of perturbed planar Hamiltonian systems}

The aim of this section is to recall one of the main technique used in the study of \textit{limit cycles} (i.e., periodic orbits that are the $\alpha-$limit set or the $\omega-$limit set of some points outside the orbit) bifurcations of planar analytic Hamiltonian dynamical systems. Note that in the case of planar systems (or dynamical systems modeled on a two-dimensional manifold), a limit cycle is an isolated periodic orbit.

We are interested in limit cycles bifurcations of planar analytic Hamiltonian dynamical systems, since by the local normal form given in Theorem \eqref{darbnf}, the (local) nontrivial dynamics of an $n-$dimensional analytic completely integrable system, is actually described by that of a planar analytic Hamiltonian dynamical system. Hence, one can use the results in the planar theory in order to analyze (locally) similar bifurcation phenomena for $n-$dimensional analytic completely integrable systems. 

In order to recall the main bifurcation result for planar Hamiltonian dynamical systems, let us introduce some notations from \cite{cli}, adapted to the real analytic category.

Let $H:D \rightarrow \mathbb{R}$ be a real analytic function defined on a domain $D\subseteq \mathbb{R}^2$, and let 
$$
X_H=\dfrac{\partial H}{\partial y}\cdot\dfrac{\partial}{\partial x} - \dfrac{\partial H}{\partial x}\cdot\dfrac{\partial}{\partial y},
$$
be the associated Hamiltonian vector field with respect to the canonical symplectic form $\omega_{\operatorname{can}}:=\mathrm{d}x\wedge\mathrm{d}y$, i.e., $\iota_{X_H}\omega_{\operatorname{can}}=\mathrm{d}H$.

Consider an analytic perturbation of the vector field $X_H$, given by the vector field 
\begin{equation}\label{phs}
X_{\varepsilon}:=X_H + \varepsilon \left(f \ \dfrac{\partial}{\partial x} + g \ \dfrac{\partial}{\partial y}\right),
\end{equation}
where $f,g:U\rightarrow \mathbb{R}$ are two real analytic functions, and $\varepsilon >0$ is a small perturbation parameter.

\medskip
\textbf{Hypothesis}: \textit{Suppose there exists a family of periodic orbits of the Hamiltonian vector filed $X_H$, $\gamma_h \subseteq H^{-1}(\{h\})$, which depends continuously on $h\in (a,b)$, for some $a,b\in\mathbb{R}$.}
\medskip

\medskip
\textbf{Problem}: \textit{Determine if there exists $h\in(a,b)$, and a periodic orbit $\Gamma_{\varepsilon}$ of the vector field $X_{\varepsilon}$, such that $\lim_{\varepsilon \rightarrow 0}\Gamma_{\varepsilon} =\gamma_{h}$, in the sense of Hausdorff distance.}
\medskip

Let us recall from \cite{cli} the classical approach of this problem, based on the Poincar\'e--Pontryagin Theorem. In order to do that, consider a segment which is transversal to each periodic orbit $\gamma_h \subseteq H^{-1}(\{h\})$, $h\in (a,b)$, of the Hamiltonian vector field $X_H$. Choosing the values of the Hamiltonian $H$ in order to parameterize the transversal segment, one denotes by $\eta_{\varepsilon}(h)$, the coordinate of the next intersection with the transverse segment, of the orbit of $X_{\varepsilon}$, starting from the point of the transverse segment with coordinate $h$. If one denotes by $\delta(h,\varepsilon):=\eta_{\varepsilon}(h) - h$, the \textit{displacement function}, then the zeros of $\delta(\cdot,\varepsilon)$ correspond to periodic orbits of the vector field $X_{\varepsilon}$.

The following result, called the Poincar\'e--Pontryagin Theorem, provides an expression of the displacement function for $\varepsilon$ approaching zero. For details regarding the proof of this theorem see, e.g., \cite{poincare}, \cite{pontryagin}, \cite{cli}.
\begin{theorem}[Poincar\'e--Pontryagin]\label{PPT} In the above hypothesis, we have that
$$
\delta(h,\varepsilon)=\varepsilon\left[I(h)+\varepsilon \varphi(h,\varepsilon)\right], \ \ \text{as} \ \  \varepsilon \rightarrow 0,
$$
where
$$
I(h):=\oint_{\gamma_{h}}-f(x,y)\mathrm{d}y + g(x,y)\mathrm{d}x,
$$
and $\varphi(h,\varepsilon)$ is a real analytic function, uniformly bounded for $(h,\varepsilon)$ in a compact region near $(h,0)$, $h\in(a,b)$.

Otherwise said, $I=I(h)$ is the first order approximation with respect to $\varepsilon$ of the displacement function $\delta=\delta(h,\varepsilon)$. 
\end{theorem}

In order to give an answer for the proposed problem, let us recall from \cite{cli} a precise definition of what is meant by bifurcation of limit cycles from a certain periodic orbit.

\begin{definition}[\cite{cli}]
In the above hypothesis, if there exists $h^{\star}\in(a,b)$ and $\varepsilon^{\star}>0$ such that the perturbed vector field $X_{\varepsilon}$ has a limit cycle $\Gamma_{\varepsilon}$, for each $0<\varepsilon<\varepsilon^{\star}$, and $\lim_{\varepsilon\rightarrow 0}\Gamma_{\varepsilon}=\gamma_{h^{\star}}$ in the sense of Hausdorff distance, then one says that the limit cycle $\Gamma_{\varepsilon}$ of the perturbed vector field $X_{\varepsilon}$, \textbf{bifurcates from the periodic orbit} $\gamma_{h^{\star}}$ of the unperturbed vector field $X_H$. 

Moreover, one say that a limit cycle $\Gamma$ of the perturbed vector field $X_{\varepsilon}$ \textbf{bifurcates form the annulus} $\bigcup_{h\in(a,b)}\gamma_h$ associated with the Hamiltonian vector field $X_H$, if there exists $h^{\prime} \in (a,b)$ such that $\Gamma$ bifurcates from the periodic orbit $\gamma_{h^{\prime}}$.
\end{definition}

\begin{theorem}[\cite{cli}]\label{teorfol}
In the above hypothesis, let us suppose moreover that $I=I(h)$ is not identically zero on $(a,b)$. Then the following statements hold true.
\begin{itemize}
\item If the perturbed vector field $X_{\varepsilon}$ has a limit cycle bifurcating from the periodic orbit $\gamma_{h^{\star}}$ of $X_H$, then $I(h^{\star})=0$.
\item If there exists $h^{\star}\in (a,b)$ a simple zero of $I$ (i.e., $I(h^{\star})=0$ and $I^{\prime}(h^{\star})\neq 0$), then the perturbed vector field $X_{\varepsilon}$ has a unique limit cycle bifurcating from the periodic orbit $\gamma_{h^{\star}}$ of $X_H$, and moreover this limit cycle is hyperbolic.
\item If there exists $h^{\star}\in (a,b)$, a zero of order $k$ of $I$ (i.e., $I(h^{\star})=I^{\prime}(h^{\star})=\dots=I^{(k-1)}(h^{\star})=0$ and $I^{(k)}(h^{\star})\neq 0$), then the perturbed vector field $X_{\varepsilon}$ has at most $k$ limit cycles (counting also the multiplicities) bifurcating from the periodic orbit $\gamma_{h^{\star}}$ of $X_H$.
\item The total number (counting also the multiplicities) of limit cycles of the perturbed vector field $X_{\varepsilon}$, bifurcating from the annulus $\bigcup_{h\in(a,b)}\gamma_h$ of $X_H$, is bounded by the maximum number (if finite) of isolated zeros (counting also the multiplicities) of $I(h)$ for $h\in (a,b)$.
\end{itemize}
\end{theorem}

\begin{remark}
In order to compute the derivative $I^{\prime}=I^{\prime}(h)$, $h\in(a,b)$, of $I(h)=\oint_{\gamma_{h}}-f(x,y)\mathrm{d}y + g(x,y)\mathrm{d}x$ (where $\gamma_h \subseteq H^{-1}(\{h\})$ is a periodic orbit of the planar Hamiltonian vector field $X_H$), one can use the \textit{Gelfand-Leray form} associated with the $2-$form $\mathrm{d}\omega$, where $\omega:= -f(x,y)\mathrm{d}y + g(x,y)\mathrm{d}x$. 

More precisely, if one denotes by $\alpha$ an analytic $1-$form such that 
\begin{equation}\label{GLf}
\mathrm{d}H\wedge\alpha =\mathrm{d}\omega,
\end{equation}
then, for any point which is not a critical point of the Hamiltonian $H$, there exists an open neighborhood on which one can define such a $1-$form $\alpha$ which verifies \eqref{GLf}. Moreover, since such a $1-$form $\alpha$ is uniquely defined modulo $\mathrm{d}H$, then the restriction of $\alpha$ to an arbitrary analytic level manifold $H^{-1}(\{h\})$ is \textbf{uniquely defined}, and is called the \textbf{Gelfand-Leray form} associated with the $2-$form $\mathrm{d}\omega$. For details regarding the Gelfand-Leray form, see e.g., \cite{arnoldd}. 

In the above notations, if $\omega$ has compact support and moreover, the support does not contains any critical point of the Hamiltonian $H$, then
$$
I^{\prime}(h)= \dfrac{\mathrm{d}}{\mathrm{d}h}\oint_{\gamma_{h}}\omega=\oint_{\gamma_{h}}\alpha.
$$ 
\end{remark}

\section{Limit cycles bifurcations of geometrically perturbed completely integrable systems}

The purpose of this section is to study from the geometric point of view the problem of limit cycles bifurcation of perturbed completely integrable systems. By geometric approach of the problem we mean the use of those perturbations which preserves some of the geometry associated with the unperturbed system. This problem was also analyzed in the geometric context for a particular three dimensional completely integrable system, namely Euler's equations of the free rigid body, \cite{bermejo1}. Also a non-geometric version of the problem, using averaging methods was recently analyzed in \cite{bermejo2}.
Our approach is a local one, and makes use of the local normal form introduced in Theorem \eqref{darbnf}. More precisely, in the first part of this section one considers the problem of limit cycles bifurcations in the context of a general type of perturbations which preserves an a-priori fixed symplectic leaf of the Poisson configuration manifold of the unperturbed completely integrable system. In the second part of this section one considers the same problem, this time one considers the class of perturbations that preserves all of the regular symplectic leaves of the Poisson configuration manifold of the unperturbed completely integrable system. Let us start now by constructing the first class of perturbed dynamical systems.

\subsection{Limit cycles bifurcations on a fixed symplectic leaf}

Let $\dot x =X(x)$, be a real analytic completely integrable system defined on a domain $\Omega\subseteq \mathbb{R}^n$. Hence, the vector field $X\in\mathfrak{X}(\Omega)$ admits $n-2$ analytic first integrals, $C_1,\dots,C_{n-2}:\Omega\rightarrow \mathbb{R}$, functionally independent almost everywhere in $\Omega$, with respect to the $n-$dimensional Lebesgue measure. Consequently, for each regular value $(c_1,\dots,c_{n-2})\in\operatorname{Im}(C_1,\dots,C_{n-2})\subseteq \mathbb{R}^{n-2}$, the corresponding real analytic manifold
\begin{equation*}
\Sigma_{c_1,\dots,c_{n-2}}:=\{(x_1,\dots,x_n)\in\Omega \mid C_1(x_1,\dots,x_n )=c_1,\ \dots \ , C_{n-2}(x_1,\dots, x_n)=c_{n-2}\},
\end{equation*}
is a dynamically invariant set of the completely integrable system generated by the vector field $X$. Recall from the second section that the completely integrable system $\dot x =X(x)$, admits a Hamiltonian realization, \eqref{systy}, with Hamiltonian $H=:C_{n-1}$, and Poisson manifold $(\Omega,\{\cdot,\cdot\}_{\nu;C_1,\dots,C_{n-2}})$. Consequently, since the functions $C_1,\dots,C_{n-2}$ form a complete set of Casimir invariants of the Poisson manifold, then the connected components of the manifold $\Sigma_{c_1,\dots,c_{n-2}}$ are two-dimensional regular symplectic leaves of the Poisson manifold $(\Omega,\{\cdot,\cdot\}_{\nu;C_1,\dots,C_{n-2}})$. For the sake of simplicity, we suppose in the sequel that $\Sigma_{c_1,\dots,c_{n-2}}$ is a connected set, and hence it is a regular symplectic leaf (if not, we shall consider instead, a connected component of $\Sigma_{c_1,\dots,c_{n-2}}$).

Let us construct now a general perturbation, $\dot x =X(x)+\varepsilon A(x)$, of the completely integrable system $\dot x =X(x)$, whose restriction to a certain open subset $\Omega^{\prime}\subseteq \Omega$, preserves the corresponding intersection with an \textit{a-priori fixed} symplectic leaf, $\Sigma_{c_1,\dots,c_{n-2}}\cap\Omega^{\prime}$. As any symplectic leaf of the Poisson manifold $(\Omega,\{\cdot,\cdot\}_{\nu;C_1,\dots,C_{n-2}})$ is kept invariant by the dynamics generated by the integrable vector field $X$, the problem reduces to find a general form of the analytic vector field $A$, whose restriction to $\Omega^{\prime}$ is tangent to the open subset $\Sigma_{c_1,\dots,c_{n-2}}\cap\Omega^{\prime}$ of the symplectic leaf $\Sigma_{c_1,\dots,c_{n-2}}$.

In order to do that, one applies a classical criterion that ensures the invariance of the real analytic manifold $\Sigma_{c_1,\dots,c_{n-2}}$ with respect to the dynamics generated by a real analytic vector field $A$. More precisely, $\Sigma_{c_1,\dots,c_{n-2}}$ is invariant with respect to the dynamics generated by the vector field $A$, if there exist some real analytic functions, $\tilde{R}_{i,j}\in\mathcal{C}^{\omega}(\Omega,\mathbb{R})$, $i,j\in\{1,\dots, n-2\}$, such that 
\begin{equation}\label{criterion}
\mathcal{L}_{A} C_i =(C_1 - c_1)\tilde{R}_{i,1} + \dots + (C_{n-2} - c_{n-2})\tilde{R}_{i,n-2},
\end{equation}
for each $i\in\{1,\dots, n-2\}$. 

\textit{Since the main part of this article is concerned with the limit cycles bifurcation problem, and our approach is done in Darboux coordinates (i.e., a set of coordinates generated by a triple $(\Omega^{\prime},\Phi_1,\Phi_2)$, consisting of an open subset $\Omega^{\prime}\subseteq \Omega$, and two analytic functions $\Phi_1,\Phi_2 \in\mathcal{C}^{\omega}(\Omega^{\prime},\mathbb{R})$, such that the map $\Phi:\Omega^{\prime}\rightarrow W^{\prime}=\Phi(\Omega^{\prime})$, given by $\Phi=(\Phi_1,\Phi_2,C_1,\dots,C_{n-2})$, is an analytic diffeomorphism), the vector field $A$ will be constructed in such a way that the perturbed vector field $\Phi_{\star}X + \varepsilon \Phi_{\star}A$ restricted to the two-dimensional subset $\Phi(\Sigma_{c_1,\dots,c_{n-2}}\cap \Omega^{\prime})\subseteq W^{\prime}$, is orbitally equivalent to a general vector field of the type \eqref{phs}}. 

Before starting the construction of $\Phi_{\star}A$, note that $\Phi(\Sigma_{c_1,\dots,c_{n-2}}\cap \Omega^{\prime})$ is exactly the symplectic leaf $\{(y_1,\dots,y_n)\in W^{\prime}\mid y_3=c_1, \dots, y_n =c_{n-2}\}$ of the Poisson manifold $\left(W^{\prime},\{\cdot,\cdot\}_{\nu_{\Phi};\Phi_{\star}C_1,\dots,\Phi_{\star}C_{n-2}}\right)$, where $\nu_{\Phi}=\Phi_{\star}\nu\cdot \Phi_{\star}\operatorname{Jac}(\Phi)$.

In order to get the expression of the perturbation vector field $\Phi_{\star}A$, one rewrites the equations \eqref{criterion} restricted to $\Omega^{\prime}$, as follows
\begin{equation}\label{criterion0}
\langle \nabla C_i, A\rangle =(C_1 - c_1){R}_{i,1}\circ \Phi + \dots + (C_{n-2} - c_{n-2}){R}_{i,n-2}\circ \Phi,
\end{equation}
or equivalently,
\begin{equation}\label{criterion2}
\langle \nabla C_i, A\rangle\circ\Phi^{-1} =(C_1 \circ \Phi^{-1} - c_1){R}_{i,1} + \dots + (C_{n-2}\circ \Phi^{-1} - c_{n-2}){R}_{i,n-2},
\end{equation}
for each $i\in\{1,\dots, n-2\}$, where ${R}_{i,j}:=\tilde{R}_{i,j}\circ \Phi^{-1}\in\mathcal{C}^{\omega}(W^{\prime},\mathbb{R})$.

\textit{Note that condition \eqref{criterion0} implies that the restriction to $\Omega^{\prime}$ of the perturbed vector field $X+\varepsilon A$, is tangent to $\Sigma_{c_1,\dots,c_{n-2}}\cap \Omega^{\prime}$.}

Let us compute now the expression of the perturbed vector field $\Phi_{\star}X + \varepsilon \Phi_{\star}A$. In order to do that, recall first from Theorem \eqref{darbnf} the expression of the Darboux normal form of the vector field $X$ written in Hamiltonian form \eqref{systy}, i.e.,
\begin{equation*}
\Phi_{\star}X = \nu_{\Phi}\cdot\left[ \dfrac{\partial(\Phi_{\star}H)}{\partial y_2}\cdot\dfrac{\partial}{\partial y_1}-\dfrac{\partial(\Phi_{\star}H)}{\partial y_1}\cdot\dfrac{\partial}{\partial y_2}\right],
\end{equation*}
where $\nu_{\Phi}=\Phi_{\star}\nu\cdot \Phi_{\star}\operatorname{Jac}(\Phi)$, and $(y_1,\dots,y_n)=\Phi(x_1,\dots,x_n)$, denote the local coordinates on the domain $W^{\prime}$.

Using the above expression of the vector field $\Phi_{\star}X$, and supposing that $\nu_{\Phi}$ is nonvanishing, one obtains successively the following equalities:
\begin{align*}
&\Phi_{\star}X +\varepsilon \Phi_{\star}A = \nu_{\Phi}\cdot\left[ \dfrac{\partial(\Phi_{\star}H)}{\partial y_2}\cdot\dfrac{\partial}{\partial y_1}-\dfrac{\partial(\Phi_{\star}H)}{\partial y_1}\cdot\dfrac{\partial}{\partial y_2}\right] + \varepsilon  \Biggl( \langle \nabla \Phi_1, A\rangle\circ\Phi^{-1}\cdot\dfrac{\partial}{\partial y_1} \\
&+ \langle \nabla \Phi_2, A\rangle\circ\Phi^{-1}\cdot\dfrac{\partial}{\partial y_2} + \langle \nabla C_1, A\rangle\circ\Phi^{-1}\cdot\dfrac{\partial}{\partial y_3} + \dots + \langle \nabla C_{n-2}, A\rangle\circ\Phi^{-1}\cdot\dfrac{\partial}{\partial y_{n}} \Biggr)\\
&= \nu_{\Phi}\cdot \left[ \dfrac{\partial(\Phi_{\star}H)}{\partial y_2} + \varepsilon \dfrac{\langle \nabla \Phi_1, A\rangle\circ\Phi^{-1}}{\nu_{\Phi}} \right]\cdot\dfrac{\partial}{\partial y_1} + \nu_{\Phi}\cdot \left[ - \dfrac{\partial(\Phi_{\star}H)}{\partial y_1} + \varepsilon \dfrac{\langle \nabla \Phi_2, A\rangle\circ\Phi^{-1}}{\nu_{\Phi}} \right]\cdot\dfrac{\partial}{\partial y_2}\\
&+ \varepsilon \langle \nabla C_1, A\rangle\circ\Phi^{-1} \cdot\dfrac{\partial}{\partial y_3} + \dots + \varepsilon \langle \nabla C_{n-2}, A\rangle\circ\Phi^{-1} \cdot\dfrac{\partial}{\partial y_n}.
\end{align*}

Since $\Phi=(\Phi_1,\Phi_2,C_1,\dots,C_{n-2})$, and $(y_1,\dots,y_n)=\Phi(x_1,\dots,x_n)$, the expression \eqref{criterion2} becomes
\begin{equation*}
\langle \nabla C_i, A\rangle(\Phi^{-1}(y_1,\dots,y_n)) =(y_3 - c_1){R}_{i,1}(y_1,\dots,y_n) + \dots + (y_n - c_{n-2}){R}_{i,n-2}(y_1,\dots,y_n),
\end{equation*}
for each $i\in\{1,\dots, n-2\}$.

Consequently, the expression of the perturbed vector field $\Phi_{\star}X + \varepsilon \Phi_{\star}A$ becomes
\begin{align*}
\Phi_{\star}X & +\varepsilon \Phi_{\star}A = \nu_{\Phi}\cdot \Biggl[ \dfrac{\partial(\Phi_{\star}H)}{\partial y_2} + \varepsilon \dfrac{\langle \nabla \Phi_1, A\rangle\circ\Phi^{-1}}{\nu_{\Phi}} \Biggr]\cdot\dfrac{\partial}{\partial y_1} \\
&+ \nu_{\Phi}\cdot \Biggl[ - \dfrac{\partial(\Phi_{\star}H)}{\partial y_1} + \varepsilon \dfrac{\langle \nabla \Phi_2, A\rangle\circ\Phi^{-1}}{\nu_{\Phi}} \Biggr]\cdot\dfrac{\partial}{\partial y_2}
+ \varepsilon \Biggl[\sum_{j=1}^{n-2}(y_{j+2}-c_{j})R_{1,j}\Biggr]\cdot\dfrac{\partial}{\partial y_3} \\
&+ \dots +  \varepsilon \Biggl[\sum_{j=1}^{n-2}(y_{j+2}-c_{j})R_{n-2,j}\Biggr]\cdot\dfrac{\partial}{\partial y_n}.
\end{align*}

Hence, the restriction to $\Phi(\Sigma_{c_1,\dots,c_{n-2}}\cap \Omega^{\prime}) = \{(y_1,\dots,y_n)\in W^{\prime} \mid y_3=c_1, \dots, y_n =c_{n-2}\}$ of the perturbed vector filed $\Phi_{\star}X + \varepsilon \Phi_{\star}A$, is orbitally equivalent to a general vector field of the type \eqref{phs}, if
\begin{equation}\label{p's}
\langle \nabla \Phi_1, A\rangle\circ\Phi^{-1}=\nu_{\Phi}P_1, \qquad \langle \nabla \Phi_2, A\rangle\circ\Phi^{-1}=\nu_{\Phi}P_2,
\end{equation}
where $P_1,P_2 \in\mathcal{C}^{\omega}(W^{\prime},\mathbb{R})$, are real analytic functions.

More precisely, the restriction to $\Phi(\Sigma_{c_1,\dots,c_{n-2}}\cap \Omega^{\prime})$ of the perturbed vector filed $\Phi_{\star}X + \varepsilon \Phi_{\star}A$, becomes
\begin{align*}
(\Phi_{\star}X +\varepsilon \Phi_{\star}A)|_{\Phi(\Sigma_{c_1,\dots,c_{n-2}}\cap \Omega^{\prime})} &= \nu_{\Phi}\cdot \Biggl[ \dfrac{\partial(\Phi_{\star}H)}{\partial y_2} + \varepsilon P_1 \Biggr] \cdot\dfrac{\partial}{\partial y_1} \\
&+ \nu_{\Phi}\cdot \Biggl[ - \dfrac{\partial(\Phi_{\star}H)}{\partial y_1} + \varepsilon P_2 \Biggr]\cdot\dfrac{\partial}{\partial y_2}.
\end{align*}

\begin{remark}\label{remaA}
Translating the relation \eqref{criterion2} and \eqref{p's} on $\Omega^{\prime}$, one obtains that the restriction to $\Omega^{\prime}$ of the vector field $A\in\mathfrak{X}(\Omega)$ is tangent to the open subset $\Sigma_{c_1,\dots,c_{n-2}}\cap\Omega^{\prime}$ of the regular symplectic leaf $\Sigma_{c_1,\dots,c_{n-2}}$ of the Poisson manifold $(\Omega,\{\cdot,\cdot\}_{\nu;C_1,\dots,C_{n-2}})$, if the following relations hold true:
\begin{align*}
\langle \nabla C_i, A|_{\Omega^{\prime}} \rangle &=(C_1 - c_1){R}_{i,1}\circ \Phi + \dots + (C_{n-2} - c_{n-2}){R}_{i,n-2}\circ \Phi, \ i\in\{1,\dots, n-2 \},\\
\langle \nabla \Phi_1, A|_{\Omega^{\prime}} \rangle &=(\nu_{\Phi}\circ\Phi)( P_1 \circ\Phi), \qquad \langle \nabla \Phi_2, A|_{\Omega^{\prime}} \rangle =(\nu_{\Phi}\circ\Phi)( P_2 \circ \Phi).
\end{align*}
\end{remark}
In order to obtain the expression of the vector field $A|_{\Omega^{\prime}}$ which verifies the above conditions, let us recall from \cite{tudoran1} the following result.
\begin{proposition}[\cite{tudoran1}]\label{formulajgp}
Let $(M,g)$ be an $n-$dimensional real analytic Riemannian manifold. Let $h_1,\dots, h_n, I_1,\dots, I_n \in\mathcal{C}^{\omega}(U,\mathbb{R})$ be a given set of real analytic functions defined on an open subset $U\subseteq M$, such that $\nabla I_1,\dots, \nabla I_n \in\mathcal{C}^{\omega}(U,\mathbb{R})$ are independent on $U$. 

Then the system of equations in the undetermined $X\in\mathfrak{X}(U)$
\begin{equation*}
g(\nabla I_1 , X) =h_1, \ \dots, \ g(\nabla I_n , X) =h_n,
\end{equation*}  
admits a unique solution, given by
\begin{equation}\label{solA}
X=\left\|\nabla I_1 \wedge \dots \wedge \nabla I_n\right\|_{n}^{-2} \cdot \sum_{i=1}^{n}(-1)^{n-i} h_i \cdot\Theta_i,
\end{equation}
where 
$$
\Theta_i := \star \left[ \bigwedge_{j=1,\ j\neq i}^{n}\nabla I_j \wedge \star\left( \bigwedge_{j=1}^{n}\nabla I_p\right)\right], \ i\in\{1,\dots,n\}.
$$  
Recall that $\star$ stands for the Hodge star operator acting on multi-vector fields, and the symbol $\nabla$ stands for the gradient operator with respect to the Riemannian metric $g$.
\end{proposition}

\begin{proposition}\label{Aformula}
The vector field $A|_{\Omega^{\prime}}\in\mathfrak{X}(\Omega^{\prime})$ fulfilling the relations from Remark \eqref{remaA} is given by:
\begin{align*}
A|_{\Omega^{\prime}}&={1/\operatorname{Jac}(\Phi)} \cdot\Biggl[ (-1)^{n-1} (\nu_{\Phi}\circ \Phi)(P_1 \circ\Phi)\cdot\Theta_{1}\\
&+ (-1)^{n-2} (\nu_{\Phi}\circ \Phi)(P_2 \circ\Phi)\cdot\Theta_{2} +\sum_{i=1}^{n-2}(-1)^{n-(i+2)}\Biggl( \sum_{j=1}^{n-2}(C_j - c_j)(R_{i,j}\circ\Phi)\Biggr)\cdot\Theta_{i+2}\Biggr],
\end{align*}
where 
\begin{align*}
\Theta_1 &=\star\Biggl( \nabla\Phi_2 \wedge \bigwedge_{j=1}^{n-2}\nabla C_{j} \Biggr),\ \ \Theta_2 =\star\Biggl( \nabla\Phi_1 \wedge \bigwedge_{j=1}^{n-2}\nabla C_{j} \Biggr),\\
\Theta_{i+2}&= \star\Biggl( \nabla\Phi_1 \wedge \nabla\Phi_2 \wedge \bigwedge_{j=1, \ j\neq i}^{n-2}\nabla C_{j} \Biggr), \ i\in\{1,\dots, n-2\}.
\end{align*}
\end{proposition}
\begin{proof}
The proof follows directly by formula \eqref{solA}, taking into account that $\Phi=(\Phi_1,\Phi_2,C_1,\dots, C_{n-2})$, and hence we have $$\star\left(\nabla\Phi_1 \wedge \nabla\Phi_2 \wedge \bigwedge_{j=1}^{n-2} \nabla C_j\right) = \operatorname{Jac}(\Phi),$$ and 
$$
\left\|\nabla\Phi_1 \wedge \nabla\Phi_2 \wedge \bigwedge_{j=1}^{n-2} \nabla C_j\right\|_{n}^{2}=\left[ \star\left(\nabla\Phi_1 \wedge \nabla\Phi_2 \wedge \bigwedge_{j=1}^{n-2} \nabla C_j\right)\right]^{2}= \left(\operatorname{Jac}(\Phi)\right)^2.
$$
\end{proof}

The results obtained so far in this subsection, can be presented in a unified manner in the following Theorem.

\begin{theorem}\label{nfcisp}
Let $$X = \sum_{i=1}^{n}\nu \cdot \dfrac{\partial(C_1,\dots,C_{n-2},x_i,H)}{\partial(x_1,\dots,x_n)}\cdot\dfrac{\partial}{\partial{x_i}},$$ be the vector field associated to the real analytic completely integrable system \eqref{sys}, written as a Hamiltonian dynamical system of the form \eqref{systy}, with Hamiltonian $H:=C_{n-1}$, defined on the Poisson manifold $\left(\Omega,\{\cdot,\cdot\}_{\nu;C_1,\dots,C_{n-2}}\right)$. 

Let $(\Omega^{\prime},\Phi_1,\Phi_2)$ be a triple consisting of an open subset  $\Omega^{\prime}\subseteq \Omega$, and two analytic functions $\Phi_1,\Phi_2 \in\mathcal{C}^{\omega}(\Omega^{\prime},\mathbb{R})$, such that the map $\Phi:\Omega^{\prime}\rightarrow W^{\prime}=\Phi(\Omega^{\prime})$, given by $\Phi=(\Phi_1,\Phi_2,C_1,\dots,C_{n-2})$, is an analytic diffeomorphism. Let $P_1,P_2,R_{i,j}\in\mathcal{C}^{\omega}(W^{\prime},\mathbb{R})$, $i,j\in\{1,\dots, n-2\}$, be arbitrary real analytic functions, and let $A\in\mathfrak{X}(\Omega)$ be a vector field whose restriction to $\Omega^{\prime}$ is defined by the formula given in Proposition \eqref{Aformula}. Then the following conclusions hold true:
\begin{itemize}
\item the push forward by $\Phi$ of the vector field $X$, is a Hamiltonian vector field with Hamiltonian $\Phi_{\star}H=\Phi_{\star}C_{n-1}$, defined on the Poisson manifold $\left(W^{\prime},\{\cdot,\cdot\}_{\nu_{\Phi};\Phi_{\star}C_1,\dots,\Phi_{\star}C_{n-2}}\right)$, having the expression
\begin{equation*}
\Phi_{\star}X = \nu_{\Phi}\cdot\left[ \dfrac{\partial(\Phi_{\star}H)}{\partial y_2}\cdot\dfrac{\partial}{\partial y_1}-\dfrac{\partial(\Phi_{\star}H)}{\partial y_1}\cdot\dfrac{\partial}{\partial y_2}\right],
\end{equation*}
where $\nu_{\Phi}=\Phi_{\star}\nu\cdot \Phi_{\star}\operatorname{Jac}(\Phi)$, and $(y_1,\dots, y_n)=\Phi(x_1,\dots, x_n)$, are the coordinates on $W^\prime$,
\item the restriction to $\Omega^{\prime}$ of the perturbed vector field $X_{\varepsilon}:= X+ \varepsilon A\in \mathfrak{X}(\Omega)$, $\varepsilon >0$, is tangent to the open subset $\Sigma_{c_1,\dots,c_{n-2}}\cap\Omega^{\prime}$ of the regular symplectic leaf $\Sigma_{c_1,\dots,c_{n-2}}=\{(x_1,\dots,x_n)\in\Omega \mid C_1(x_1,\dots,x_n )=c_1,\ \dots \ , C_{n-2}(x_1,\dots, x_n)=c_{n-2}\}$ of the Poisson manifold $(\Omega,\{\cdot,\cdot\}_{\nu;C_1,\dots,C_{n-2}})$; moreover, $\Phi(\Sigma_{c_1,\dots,c_{n-2}}\cap \Omega^{\prime})$ is exactly the symplectic leaf $\{(y_1,\dots,y_n)\in W^{\prime}\mid y_3=c_1, \dots, y_n =c_{n-2}\}$ of the Poisson manifold $\left(W^{\prime},\{\cdot,\cdot\}_{\nu_{\Phi};\Phi_{\star}C_1,\dots,\Phi_{\star}C_{n-2}}\right)$,
\item the push forward by $\Phi$ of the vector field $X_{\varepsilon}$, $\Phi_{\star}X_{\varepsilon}\in\mathfrak{X}(W^\prime)$, is given by
\begin{align*}
\Phi_{\star}X_{\varepsilon} &= \nu_{\Phi}\cdot \Biggl[ \dfrac{\partial(\Phi_{\star}H)}{\partial y_2} + \varepsilon P_1 \Biggr]\cdot\dfrac{\partial}{\partial y_1} + \nu_{\Phi}\cdot \Biggl[ - \dfrac{\partial(\Phi_{\star}H)}{\partial y_1} + \varepsilon P_2 \Biggr]\cdot\dfrac{\partial}{\partial y_2}\\
&+ \varepsilon \Biggl[\sum_{j=1}^{n-2}(y_{j+2}-c_{j})R_{1,j}\Biggr]\cdot\dfrac{\partial}{\partial y_3} + \dots +  \varepsilon \Biggl[\sum_{j=1}^{n-2}(y_{j+2}-c_{j})R_{n-2,j}\Biggr]\cdot\dfrac{\partial}{\partial y_n},
\end{align*}
\item the restriction to $\Phi(\Sigma_{c_1,\dots,c_{n-2}}\cap \Omega^{\prime})$ of the perturbed vector filed $\Phi_{\star}X_{\varepsilon}$, has the expression
\begin{equation*}
\Phi_{\star}X_{\varepsilon}|_{\Phi(\Sigma_{c_1,\dots,c_{n-2}}\cap \Omega^{\prime})} = \nu_{\Phi}\cdot \Biggl[ \dfrac{\partial(\Phi_{\star}H)}{\partial y_2} + \varepsilon P_1 \Biggr]\cdot\dfrac{\partial}{\partial y_1} + \nu_{\Phi}\cdot \Biggl[ - \dfrac{\partial(\Phi_{\star}H)}{\partial y_1} + \varepsilon P_2 \Biggr]\cdot\dfrac{\partial}{\partial y_2}.
\end{equation*}
\end{itemize}
\end{theorem}

Let us return now to the main problem of this article, namely, the bifurcation of limit cycles of perturbed integrable dynamical systems. In order to state the main result of this section, we need to fix some technical details.

Let
\begin{equation}\label{syscis}
X = \sum_{i=1}^{n}\nu \cdot \dfrac{\partial(C_1,\dots,C_{n-2},x_i,H)}{\partial(x_1,\dots,x_n)}\cdot\dfrac{\partial}{\partial{x_i}},
\end{equation}
be the vector field associated to the real analytic completely integrable system \eqref{sys}, written as a Hamiltonian dynamical system of the form \eqref{systy}, with Hamiltonian $H:=C_{n-1}$, defined on the Poisson manifold $\left(\Omega,\{\cdot,\cdot\}_{\nu;C_1,\dots,C_{n-2}}\right)$. 

Let $(\Omega^{\prime},\Phi_1,\Phi_2)$ be a triple consisting of an open subset  $\Omega^{\prime}\subseteq \Omega$, and two analytic functions $\Phi_1,\Phi_2 \in\mathcal{C}^{\omega}(\Omega^{\prime},\mathbb{R})$, such that the map $\Phi:\Omega^{\prime}\rightarrow W^{\prime}=\Phi(\Omega^{\prime})$, given by $\Phi=(\Phi_1,\Phi_2,C_1,\dots,C_{n-2})$, is an analytic diffeomorphism. The choices of the open set $\Omega^{\prime}$, and respectively the analytic functions $\Phi_1,\Phi_2 \in\mathcal{C}^{\omega}(\Omega^{\prime},\mathbb{R})$, are made in accordance with the location of the periodic orbits (if any) of the completely integrable system \eqref{syscis}. Let us fix a regular symplectic leaf, $\Sigma_{c_1 ,\dots, c_{n-2}}=C_{1}^{-1}(\{c_1 \})\cap\dots\cap C_{n-2}^{-1}(\{c_{n-2}\})$, of the Poisson manifold $\left(\Omega,\{\cdot,\cdot\}_{\nu;C_1,\dots,C_{n-2}}\right)$, and assume for the sake of simplicity that $\Sigma_{c_1 ,\dots, c_{n-2}}\cap\Omega^{\prime}$ is a connected subset of $\Sigma_{c_1 ,\dots, c_{n-2}}$ (if not, we shall consider instead a connected component of $\Sigma_{c_1 ,\dots, c_{n-2}}\cap\Omega^{\prime}$).

\medskip
\textbf{Hypothesis (H):} \textit{Suppose there exists a family of periodic orbits of the completely integrable system generated by the vector field \eqref{syscis}, $\gamma_{h} \subseteq H|_{\Sigma_{c_1 ,\dots, c_{n-2}}\cap\Omega^{\prime}}^{-1}(\{h\})$ (which depends continuously on $h\in (a_{c_1 ,\dots, c_{n-2}},b_{c_1 ,\dots, c_{n-2}})$, for some $a_{c_1 ,\dots, c_{n-2}},b_{c_1 ,\dots, c_{n-2}}\in\mathbb{R}$), located on the open subset $\Sigma_{c_1 ,\dots, c_{n-2}}\cap\Omega^{\prime}$ of a regular symplectic leaf, $\Sigma_{c_1 ,\dots, c_{n-2}}=C_{1}^{-1}(\{c_1 \})\cap\dots\cap C_{n-2}^{-1}(\{c_{n-2}\})$, of the Poisson manifold $\left(\Omega,\{\cdot,\cdot\}_{\nu;C_1,\dots,C_{n-2}}\right)$.} 
\medskip

\textit{Note that the real numbers $a_{c_1 ,\dots, c_{n-2}}< b_{c_1 ,\dots, c_{n-2}}$ depend also on the open set $\Omega^{\prime}$; the reason we do not specify their dependence on $\Omega^{\prime}$ explicitly, is that $\Omega^{\prime}$ is supposed to be the largest subset of $\Omega$ where the Darboux normal form holds.}

\medskip
Then, by Theorem \eqref{nfcisp} we obtain that $\Phi(\gamma_h)\subseteq (\Phi_{\star}H)|_{\Phi(\Sigma_{c_1 ,\dots, c_{n-2}}\cap\Omega^{\prime})}^{-1}(\{h\})$ (depending continuously on $h\in (a_{c_1 ,\dots, c_{n-2}},b_{c_1 ,\dots, c_{n-2}})$), are periodic orbits of the Hamilton-Poisson vector field $\Phi_{\star}X\in\mathfrak{X}(W^{\prime})$, located on the two-dimensional symplectic leaf $\{(y_1,\dots,y_n)\in W^{\prime}\mid y_3=c_1, \dots, y_n =c_{n-2}\}$ of the Poisson manifold $\left(W^{\prime},\{\cdot,\cdot\}_{\nu_{\Phi};\Phi_{\star}C_1,\dots,\Phi_{\star}C_{n-2}}\right)$. 

Moreover, recall also from Theorem \eqref{nfcisp} that the restriction to the symplectic leaf $\Phi(\Sigma_{c_1,\dots,c_{n-2}}\cap \Omega^{\prime})=\{(y_1,\dots,y_n)\in W^{\prime}\mid y_3=c_1, \dots, y_n =c_{n-2}\}$, of the perturbed vector filed $\Phi_{\star}X_{\varepsilon}$ (where $X_{\varepsilon}=X+\varepsilon A$, with $A$ defined in Proposition \eqref{Aformula}), has the expression
\begin{equation}\label{systreduced}
\Phi_{\star}X_{\varepsilon}|_{\Phi(\Sigma_{c_1,\dots,c_{n-2}}\cap \Omega^{\prime})} = \nu_{\Phi}\cdot \Biggl[ \dfrac{\partial(\Phi_{\star}H)}{\partial y_2} + \varepsilon P_1 \Biggr]\cdot\dfrac{\partial}{\partial y_1} + \nu_{\Phi}\cdot \Biggl[ - \dfrac{\partial(\Phi_{\star}H)}{\partial y_1} + \varepsilon P_2 \Biggr]\cdot\dfrac{\partial}{\partial y_2}.
\end{equation}
Following \eqref{PPT}, for each $h\in(a_{c_1 ,\dots, c_{n-2}},b_{c_1 ,\dots, c_{n-2}})$, we define
\begin{equation*}
I_{c_1,\dots,c_{n-2}}(h):=\oint_{\Phi(\gamma_{h})}-P_{1}(y_1,y_2,c_1,\dots,c_{n-2})\mathrm{d}y_2 + P_{2}(y_1,y_2,c_1,\dots,c_{n-2})\mathrm{d}y_1.
\end{equation*}

Now one can state the main result of this section, which provides a criterion that guarantees the existence of bifurcations of limit cycles of a perturbed vector field $X_{\varepsilon}=X+\varepsilon A \in\mathfrak{X}(\Omega)$ (with $A|_{\Omega^{\prime}}$ defined as in Proposition \eqref{Aformula}), located on the domain $\Sigma_{c_1 ,\dots, c_{n-2}}\cap\Omega^{\prime}$ of the regular symplectic leaf $\Sigma_{c_1 ,\dots, c_{n-2}}=C_{1}^{-1}(\{c_1 \})\cap\dots\cap C_{n-2}^{-1}(\{c_{n-2}\})$, of the  Poisson configuration manifold $\left(\Omega,\{\cdot,\cdot\}_{\nu;C_1,\dots,C_{n-2}}\right)$ of the completely integrable system \eqref{syscis}. Recall  from Theorem \eqref{nfcisp} that the orbits of the vector field $X_{\varepsilon}$, located on the domain $\Sigma_{c_1 ,\dots, c_{n-2}}\cap\Omega^{\prime}$, are in one-to-one correspondence with the orbits of the vector field $\Phi_{\star}X_{\varepsilon}$,  located on the domain $\Phi(\Sigma_{c_1 ,\dots, c_{n-2}}\cap\Omega^{\prime})$.
\medskip

\begin{theorem}\label{MainThm1}
In the above settings, assume that $I_{c_1,\dots,c_{n-2}}=I_{c_1,\dots,c_{n-2}}(h)$ is not identically zero on $(a_{c_1 ,\dots, c_{n-2}},b_{c_1 ,\dots, c_{n-2}})$. Then the following statements hold true.
\begin{itemize}
\item If the perturbed vector field $X_{\varepsilon}$ has a limit cycle on the domain $\Sigma_{c_1 ,\dots, c_{n-2}}\cap\Omega^{\prime}$ of the regular symplectic leaf $\Sigma_{c_1 ,\dots, c_{n-2}}$, bifurcating from the periodic orbit $\gamma_{h^{\star}}$ of the completely integrable vector field $X$, then $I_{c_1,\dots,c_{n-2}}(h^{\star})=0$.
\item If there exists $h^{\star}\in (a_{c_1 ,\dots, c_{n-2}},b_{c_1 ,\dots, c_{n-2}})$ a simple zero of $I_{c_1,\dots,c_{n-2}}$, then the perturbed vector field $X_{\varepsilon}$ has a unique limit cycle on the domain $\Sigma_{c_1 ,\dots, c_{n-2}}\cap\Omega^{\prime}$ of the regular symplectic leaf $\Sigma_{c_1 ,\dots, c_{n-2}}$, bifurcating from the periodic orbit $\gamma_{h^{\star}}$ of the completely integrable vector field $X$, and moreover, this limit cycle is hyperbolic.
\item If there exists $h^{\star}\in (a_{c_1 ,\dots, c_{n-2}},b_{c_1 ,\dots, c_{n-2}})$, a zero of order $k$ of $I_{c_1,\dots,c_{n-2}}$, then the perturbed vector field $X_{\varepsilon}$ has at most $k$ limit cycles (counting also the multiplicities) on the domain $\Sigma_{c_1 ,\dots, c_{n-2}}\cap\Omega^{\prime}$ of the regular symplectic leaf $\Sigma_{c_1 ,\dots, c_{n-2}}$, bifurcating from the periodic orbit $\gamma_{h^{\star}}$ of the completely integrable vector field $X$.
\item The total number (counting also the multiplicities) of limit cycles of the perturbed vector field $X_{\varepsilon}$, on the domain $\Sigma_{c_1 ,\dots, c_{n-2}}\cap\Omega^{\prime}$ of the regular symplectic leaf $\Sigma_{c_1 ,\dots, c_{n-2}}$, bifurcating from the annulus $\bigcup_{h\in(a,b)}\gamma_h$ of $X$, is bounded by the maximum number (if finite) of isolated zeros (counting also the multiplicities) of $I_{c_1,\dots,c_{n-2}}(h)$ for $h\in (a_{c_1 ,\dots, c_{n-2}},b_{c_1 ,\dots, c_{n-2}})$.
\end{itemize}
\end{theorem}
\begin{proof}
Recall first that the orbits of the vector field $X_{\varepsilon}$, located on the domain $\Sigma_{c_1 ,\dots, c_{n-2}}\cap\Omega^{\prime}$, are in one-to-one correspondence with the orbits of the vector field $\Phi_{\star}X_{\varepsilon}$, located on the domain $\Phi(\Sigma_{c_1 ,\dots, c_{n-2}}\cap\Omega^{\prime})$. 

Consequently, it is enough to show the bifurcation result for the vector field $\Phi_{\star}X_{\varepsilon}$, and the corresponding family of periodic orbits,  
$\Phi(\gamma_h)\subseteq (\Phi_{\star}H)|_{\Phi(\Sigma_{c_1 ,\dots, c_{n-2}}\cap\Omega^{\prime})}^{-1}(\{h\})$ (depending continuously on $h\in (a_{c_1 ,\dots, c_{n-2}},b_{c_1 ,\dots, c_{n-2}})$), located on the two-dimensional symplectic leaf $\Phi(\Sigma_{c_1 ,\dots, c_{n-2}}\cap\Omega^{\prime})=\{(y_1,\dots,y_n)\in W^{\prime}\mid y_3=c_1, \dots, y_n =c_{n-2}\}$ of the Poisson manifold$\left(W^{\prime},\{\cdot,\cdot\}_{\nu_{\Phi};\Phi_{\star}C_1,\dots,\Phi_{\star}C_{n-2}}\right)$.  

The rest of the proof follows now from Theorem \eqref{teorfol}. Indeed, in order to apply the Theorem \eqref{teorfol}, recall from Theorem \eqref{nfcisp} that the restriction to the symplectic leaf $\Phi(\Sigma_{c_1,\dots,c_{n-2}}\cap \Omega^{\prime})=\{(y_1,\dots,y_n)\in W^{\prime}\mid y_3=c_1, \dots, y_n =c_{n-2}\}$ of the perturbed vector filed $\Phi_{\star}X_{\varepsilon}$, has the expression
\begin{equation*}
\Phi_{\star}X_{\varepsilon}|_{\Phi(\Sigma_{c_1,\dots,c_{n-2}}\cap \Omega^{\prime})} = \nu_{\Phi}\cdot \Biggl[ \dfrac{\partial(\Phi_{\star}H)}{\partial y_2} + \varepsilon P_1 \Biggr]\cdot\dfrac{\partial}{\partial y_1} + \nu_{\Phi}\cdot \Biggl[ - \dfrac{\partial(\Phi_{\star}H)}{\partial y_1} + \varepsilon P_2 \Biggr]\cdot\dfrac{\partial}{\partial y_2}.
\end{equation*}
Hence, since $\nu_{\Phi}$ is nonzero and sign invariant, the vector field $\Phi_{\star}X_{\varepsilon}|_{\Phi(\Sigma_{c_1,\dots,c_{n-2}}\cap \Omega^{\prime})}$ is orbitally equivalent to the two-dimensional vector field $\mathbb{X}_{\varepsilon}\in\mathfrak{X}(\Phi(\Sigma_{c_1,\dots,c_{n-2}}\cap \Omega^{\prime}))$, defined by
$$
\mathbb{X}_{\varepsilon}:= X_{\Phi_{\star}H}+ \varepsilon \left( P_1 \dfrac{\partial}{\partial y_1} + P_2 \dfrac{\partial}{\partial y_2}\right),
$$
where the Hamiltonian vector field $X_{\Phi_{\star}H} \in \mathfrak{X}(\Phi(\Sigma_{c_1,\dots,c_{n-2}}\cap \Omega^{\prime}))$ has the expression
$$
X_{\Phi_{\star}H} = \dfrac{\partial(\Phi_{\star}H)}{\partial y_2} \cdot\dfrac{\partial}{\partial y_1} - \dfrac{\partial(\Phi_{\star}H)}{\partial y_1}\cdot\dfrac{\partial}{\partial y_2},
$$
\begin{align*}
\Phi_{\star}H =(\Phi_{\star}H)|_{\Phi(\Sigma_{c_1,\dots,c_{n-2}}\cap \Omega^{\prime})}=(\Phi_{\star} &H) (y_1,y_2,c_1,\dots,c_{n-2}),\\
& (\forall) (y_1,y_2,c_1,\dots,c_{n-2})\in \Phi(\Sigma_{c_1,\dots,c_{n-2}}\cap \Omega^{\prime}).
\end{align*}
\end{proof}

\subsection{Limit cycles bifurcations of perturbed completely integrable systems preserving the regular part of the symplectic foliation}

Let $\dot x =X(x)$, be a real analytic completely integrable system defined on a domain $\Omega\subseteq \mathbb{R}^n$. Hence, the vector field $X\in\mathfrak{X}(\Omega)$ admits $n-2$ analytic first integrals, $C_1,\dots,C_{n-2}:\Omega\rightarrow \mathbb{R}$, functionally independent almost everywhere in $\Omega$, with respect to the $n-$dimensional Lebesgue measure. Consequently, for each regular value $(c_1,\dots,c_{n-2})\in\operatorname{Im}(C_1,\dots,C_{n-2})\subseteq \mathbb{R}^{n-2}$, the corresponding real analytic manifold
\begin{equation*}
\Sigma_{c_1,\dots,c_{n-2}}:=\{(x_1,\dots,x_n)\in\Omega \mid C_1(x_1,\dots,x_n )=c_1,\ \dots \ , C_{n-2}(x_1,\dots, x_n)=c_{n-2}\},
\end{equation*}
is a dynamically invariant set of the completely integrable system generated by the vector field $X$. Recall from the second section of this article that the completely integrable system $\dot x =X(x)$, admits the Hamiltonian realization \eqref{systy} on the Poisson manifold $(\Omega,\{\cdot,\cdot\}_{\nu;C_1,\dots,C_{n-2}})$, given by the Hamiltonian function $H=:C_{n-1}$.

Let us construct now a general perturbation, $\dot x =X(x)+\varepsilon A(x)$, $\varepsilon >0$, of the completely integrable system $\dot x =X(x)$, whose restriction to a certain open subset $\Omega^{\prime}\subseteq \Omega$, preserves the regular part of the symplectic foliation of the Poisson manifold $(\Omega^{\prime},\{\cdot,\cdot\}_{\nu;C_1,\dots,C_{n-2}})$.

In order to do that, the restrictions to the open subset $\Omega^{\prime}\subseteq \Omega$ of the Casimir functions $C_1,\dots,C_{n-2}$, need to be integrals of motion of the perturbed vector field $X_{\varepsilon}:=X+\varepsilon A$ restricted to $\Omega^{\prime}$. As $C_1,\dots,C_{n-2}$ are integrals of motion of $X$, the above condition is equivalent to hold true only for the vector field $A$ restricted to $\Omega^{\prime}$.

\textit{Since the main part of this article is concerned with the limit cycles bifurcation problem, and our approach is done in Darboux coordinates (i.e., a set of coordinates generated by a triple $(\Omega^{\prime},\Phi_1,\Phi_2)$, consisting of an open subset $\Omega^{\prime}\subseteq \Omega$, and two analytic functions $\Phi_1,\Phi_2 \in\mathcal{C}^{\omega}(\Omega^{\prime},\mathbb{R})$, such that the map $\Phi:\Omega^{\prime}\rightarrow W^{\prime}=\Phi(\Omega^{\prime})$, given by $\Phi=(\Phi_1,\Phi_2,C_1,\dots,C_{n-2})$, is an analytic diffeomorphism), then the expression of the vector field $A$ will be constructed in such a way that the perturbed vector field $\Phi_{\star}X + \varepsilon \Phi_{\star}A$ restricted to each two-dimensional subset $\Phi(\Sigma_{c_1,\dots,c_{n-2}}\cap \Omega^{\prime})\subseteq W^{\prime}$, will be orbitally equivalent to a general vector field of the type \eqref{phs}}. 

Before starting the construction of $\Phi_{\star}A$, note that $\Phi(\Sigma_{c_1,\dots,c_{n-2}}\cap \Omega^{\prime})$ is exactly the symplectic leaf $\{(y_1,\dots,y_n)\in W^{\prime}\mid y_3=c_1, \dots, y_n =c_{n-2}\}$ of the Poisson manifold $\left(W^{\prime},\{\cdot,\cdot\}_{\nu_{\Phi};\Phi_{\star}C_1,\dots,\Phi_{\star}C_{n-2}}\right)$, where $\nu_{\Phi}=\Phi_{\star}\nu\cdot \Phi_{\star}\operatorname{Jac}(\Phi)$.

Let us compute now the expression of the perturbed vector field $\Phi_{\star}X + \varepsilon \Phi_{\star}A$. In order to do that, recall first from Theorem \eqref{darbnf} the expression of the Darboux normal form of the vector field $X$ written in Hamiltonian form \eqref{systy}, i.e.,
\begin{equation*}
\Phi_{\star}X = \nu_{\Phi}\cdot\left[ \dfrac{\partial(\Phi_{\star}H)}{\partial y_2}\cdot\dfrac{\partial}{\partial y_1}-\dfrac{\partial(\Phi_{\star}H)}{\partial y_1}\cdot\dfrac{\partial}{\partial y_2}\right],
\end{equation*}
where $\nu_{\Phi}=\Phi_{\star}\nu\cdot \Phi_{\star}\operatorname{Jac}(\Phi)$, and $(y_1,\dots,y_n)=\Phi(x_1,\dots,x_n)$ denote the local coordinates on the domain $W^{\prime}$.

Using the above expression of the vector field $\Phi_{\star}X$, and supposing that $\nu_{\Phi}$ is nonvanishing, one obtains successively the following equalities
\begin{align*}
&\Phi_{\star}X +\varepsilon \Phi_{\star}A = \nu_{\Phi}\cdot\left[ \dfrac{\partial(\Phi_{\star}H)}{\partial y_2}\cdot\dfrac{\partial}{\partial y_1}-\dfrac{\partial(\Phi_{\star}H)}{\partial y_1}\cdot\dfrac{\partial}{\partial y_2}\right] + \varepsilon  \Biggl( \langle \nabla \Phi_1, A\rangle\circ\Phi^{-1}\cdot\dfrac{\partial}{\partial y_1} \\
&+ \langle \nabla \Phi_2, A\rangle\circ\Phi^{-1}\cdot\dfrac{\partial}{\partial y_2} + \langle \nabla C_1, A\rangle\circ\Phi^{-1}\cdot\dfrac{\partial}{\partial y_3} + \dots + \langle \nabla C_{n-2}, A\rangle\circ\Phi^{-1}\cdot\dfrac{\partial}{\partial y_{n}} \Biggr)\\
&= \nu_{\Phi}\cdot \left[ \dfrac{\partial(\Phi_{\star}H)}{\partial y_2} + \varepsilon \dfrac{\langle \nabla \Phi_1, A\rangle\circ\Phi^{-1}}{\nu_{\Phi}} \right]\cdot\dfrac{\partial}{\partial y_1} + \nu_{\Phi}\cdot \left[ - \dfrac{\partial(\Phi_{\star}H)}{\partial y_1} + \varepsilon \dfrac{\langle \nabla \Phi_2, A\rangle\circ\Phi^{-1}}{\nu_{\Phi}} \right]\cdot\dfrac{\partial}{\partial y_2}\\
&+ \varepsilon \langle \nabla C_1, A\rangle\circ\Phi^{-1} \cdot\dfrac{\partial}{\partial y_3} + \dots + \varepsilon \langle \nabla C_{n-2}, A\rangle\circ\Phi^{-1} \cdot\dfrac{\partial}{\partial y_n}.
\end{align*}

Since $C_1,\dots,C_{n-2}\in\mathcal{C}^{\omega}(\Omega^{\prime},\mathbb{R})$ need to be first integrals of the restriction to $\Omega^{\prime}$ of the vector field $A$, we have that
\begin{equation}\label{criterion00}
\langle \nabla C_i, A\rangle = 0, \ \text{for each} \  i\in\{1,\dots, n-2\}. 
\end{equation}

Consequently, the expression of the perturbed vector field $\Phi_{\star}X + \varepsilon \Phi_{\star}A$ becomes
\begin{align*}
\Phi_{\star}X  +\varepsilon \Phi_{\star}A &= \nu_{\Phi}\cdot \Biggl[ \dfrac{\partial(\Phi_{\star}H)}{\partial y_2} + \varepsilon \dfrac{\langle \nabla \Phi_1, A\rangle\circ\Phi^{-1}}{\nu_{\Phi}} \Biggr]\cdot\dfrac{\partial}{\partial y_1} \\
&+ \nu_{\Phi}\cdot \Biggl[ - \dfrac{\partial(\Phi_{\star}H)}{\partial y_1} + \varepsilon \dfrac{\langle \nabla \Phi_2, A\rangle\circ\Phi^{-1}}{\nu_{\Phi}} \Biggr]\cdot\dfrac{\partial}{\partial y_2}.
\end{align*}

Hence, the restriction of the perturbed vector filed $\Phi_{\star}X + \varepsilon \Phi_{\star}A$ to a general regular symplectic leaf $\Phi(\Sigma_{c_1,\dots,c_{n-2}}\cap \Omega^{\prime})$, is orbitally equivalent to a general vector field of the type \eqref{phs}, if
\begin{equation}\label{q's}
\langle \nabla \Phi_1, A\rangle\circ\Phi^{-1}=\nu_{\Phi}Q_1, \qquad \langle \nabla \Phi_2, A\rangle\circ\Phi^{-1}=\nu_{\Phi}Q_2,
\end{equation}
where $Q_1,Q_2 \in\mathcal{C}^{\omega}(W^{\prime},\mathbb{R})$, are real analytic functions.

More precisely, the restriction of the perturbed vector filed $\Phi_{\star}X + \varepsilon \Phi_{\star}A$ to a general regular symplectic leaf $\Phi(\Sigma_{c_1,\dots,c_{n-2}}\cap \Omega^{\prime})$, becomes
\begin{align*}
(\Phi_{\star}X +\varepsilon \Phi_{\star}A)|_{\Phi(\Sigma_{c_1,\dots,c_{n-2}}\cap \Omega^{\prime})} &= \nu_{\Phi}\cdot \Biggl[ \dfrac{\partial(\Phi_{\star}H)}{\partial y_2} + \varepsilon Q_1 \Biggr] \cdot\dfrac{\partial}{\partial y_1} \\
&+ \nu_{\Phi}\cdot \Biggl[ - \dfrac{\partial(\Phi_{\star}H)}{\partial y_1} + \varepsilon Q_2 \Biggr]\cdot\dfrac{\partial}{\partial y_2}.
\end{align*}

\begin{remark}\label{remaAB}
The relations \eqref{criterion00} and \eqref{q's} say that, the restriction to $\Omega^{\prime}$ of the vector field $A\in\mathfrak{X}(\Omega)$ preserves the regular part of the symplectic foliation of the Poisson manifold $(\Omega^{\prime},\{\cdot,\cdot\}_{\nu;C_1,\dots,C_{n-2}})$, and moreover, the dynamics restricted to a general regular symplectic leaf is orbitally equivalent to a dynamics of the type \eqref{phs}, if the following relations hold true:
\begin{align*}
\langle \nabla C_i, A|_{\Omega^{\prime}} \rangle &=0, \ i\in\{1,\dots, n-2 \},\\
\langle \nabla \Phi_1, A|_{\Omega^{\prime}} \rangle &=(\nu_{\Phi}\circ\Phi)( Q_1 \circ\Phi), \qquad \langle \nabla \Phi_2, A|_{\Omega^{\prime}} \rangle =(\nu_{\Phi}\circ\Phi)( Q_2 \circ \Phi).
\end{align*}
\end{remark}

In order to obtain the expression of the vector field $A|_{\Omega^{\prime}}$ which verifies the above conditions, we use again the formula given in Proposition \eqref{formulajgp}.

\begin{proposition}\label{AformulaB}
The vector field $A|_{\Omega^{\prime}}\in\mathfrak{X}(\Omega^{\prime})$ fulfilling the relations from Remark \eqref{remaAB} is given by:
\begin{align*}
A|_{\Omega^{\prime}}&={1/\operatorname{Jac}(\Phi)} \cdot\Biggl[ (-1)^{n-1} (\nu_{\Phi}\circ \Phi)(Q_1 \circ\Phi)\cdot\Theta_{1} + (-1)^{n-1} (\nu_{\Phi}\circ \Phi)(Q_2 \circ\Phi)\cdot\Theta_{2}\Biggr],
\end{align*}
where 
\begin{align*}
\Theta_1 &=\star\Biggl( \nabla\Phi_2 \wedge \bigwedge_{j=1}^{n-2}\nabla C_{j} \Biggr),\ \ \Theta_2 =\star\Biggl( \nabla\Phi_1 \wedge \bigwedge_{j=1}^{n-2}\nabla C_{j} \Biggr).
\end{align*}
\end{proposition}
\begin{proof}
The proof follows directly by formula \eqref{solA}, taking into account that $\Phi=(\Phi_1,\Phi_2,C_1,\dots, C_{n-2})$, and hence we have $$\star\left(\nabla\Phi_1 \wedge \nabla\Phi_2 \wedge \bigwedge_{j=1}^{n-2} \nabla C_j\right) = \operatorname{Jac}(\Phi),$$ and
$$
\left\|\nabla\Phi_1 \wedge \nabla\Phi_2 \wedge \bigwedge_{j=1}^{n-2} \nabla C_j\right\|_{n}^{2}=\left[ \star\left(\nabla\Phi_1 \wedge \nabla\Phi_2 \wedge \bigwedge_{j=1}^{n-2} \nabla C_j\right)\right]^{2}= \left(\operatorname{Jac}(\Phi)\right)^2.
$$
\end{proof}

The results obtained so far in this subsection, can be presented in a unified manner in the following Theorem.

\begin{theorem}\label{nfcisp1}
Let $$X = \sum_{i=1}^{n}\nu \cdot \dfrac{\partial(C_1,\dots,C_{n-2},x_i,H)}{\partial(x_1,\dots,x_n)}\cdot\dfrac{\partial}{\partial{x_i}},$$ be the vector field associated to the real analytic completely integrable system \eqref{sys}, written as a Hamiltonian dynamical system of the form \eqref{systy}, with Hamiltonian $H:=C_{n-1}$, defined on the Poisson manifold $\left(\Omega,\{\cdot,\cdot\}_{\nu;C_1,\dots,C_{n-2}}\right)$. 

Let $(\Omega^{\prime},\Phi_1,\Phi_2)$ be a triple consisting of an open subset  $\Omega^{\prime}\subseteq \Omega$, and two analytic functions $\Phi_1,\Phi_2 \in\mathcal{C}^{\omega}(\Omega^{\prime},\mathbb{R})$, such that the map $\Phi:\Omega^{\prime}\rightarrow W^{\prime}=\Phi(\Omega^{\prime})$, given by $\Phi=(\Phi_1,\Phi_2,C_1,\dots,C_{n-2})$, is an analytic diffeomorphism. Let $Q_1,Q_2\in\mathcal{C}^{\omega}(W^{\prime},\mathbb{R})$, be arbitrary real analytic functions, and let $A\in\mathfrak{X}(\Omega)$ be a vector field whose restriction to $\Omega^{\prime}$ is defined by the formula given in Proposition \eqref{AformulaB}. Then the following conclusions hold true:
\begin{itemize}
\item the push forward by $\Phi$ of the vector field $X$, is a Hamiltonian vector field with Hamiltonian $\Phi_{\star}H=\Phi_{\star}C_{n-1}$, defined on the Poisson manifold $\left(W^{\prime},\{\cdot,\cdot\}_{\nu_{\Phi};\Phi_{\star}C_1,\dots,\Phi_{\star}C_{n-2}}\right)$, having the expression
\begin{equation*}
\Phi_{\star}X = \nu_{\Phi}\cdot\left[ \dfrac{\partial(\Phi_{\star}H)}{\partial y_2}\cdot\dfrac{\partial}{\partial y_1}-\dfrac{\partial(\Phi_{\star}H)}{\partial y_1}\cdot\dfrac{\partial}{\partial y_2}\right],
\end{equation*}
where $\nu_{\Phi}=\Phi_{\star}\nu\cdot \Phi_{\star}\operatorname{Jac}(\Phi)$, and $(y_1,\dots, y_n)=\Phi(x_1,\dots, x_n)$, are the coordinates on $W^\prime$,
\item the restriction to $\Omega^{\prime}$ of the perturbed vector field $X_{\varepsilon}:= X+ \varepsilon A\in \mathfrak{X}(\Omega)$, $\varepsilon >0$, is tangent to the regular part of the symplectic foliation of the Poisson manifold $(\Omega^{\prime},\{\cdot,\cdot\}_{\nu;C_1,\dots,C_{n-2}})$,
\item the push forward by $\Phi$ of the vector field $X_{\varepsilon}$, $\Phi_{\star}X_{\varepsilon}\in\mathfrak{X}(W^\prime)$, is given by
\begin{align*}
\Phi_{\star}X_{\varepsilon} = \nu_{\Phi}\cdot \Biggl[ \dfrac{\partial(\Phi_{\star}H)}{\partial y_2} + \varepsilon Q_1 \Biggr]\cdot\dfrac{\partial}{\partial y_1} + \nu_{\Phi}\cdot \Biggl[ - \dfrac{\partial(\Phi_{\star}H)}{\partial y_1} + \varepsilon Q_2 \Biggr]\cdot\dfrac{\partial}{\partial y_2},
\end{align*}
and is tangent to the regular part of the symplectic foliation of the Poisson manifold $\left(W^{\prime},\{\cdot,\cdot\}_{\nu_{\Phi};\Phi_{\star}C_1,\dots,\Phi_{\star}C_{n-2}}\right)$.
\end{itemize}
\end{theorem}

Let us now return to the main problem of this article, namely, the bifurcation of limit cycles of perturbed integrable dynamical systems. In order to state the main result of this section, we need to fix some technical details.

Let
\begin{equation}\label{syscis1}
X = \sum_{i=1}^{n}\nu \cdot \dfrac{\partial(C_1,\dots,C_{n-2},x_i,H)}{\partial(x_1,\dots,x_n)}\cdot\dfrac{\partial}{\partial{x_i}},
\end{equation}
be the vector field associated to the real analytic completely integrable system \eqref{sys}, written as a Hamiltonian dynamical system of the form \eqref{systy}, with Hamiltonian $H:=C_{n-1}$, defined on the Poisson manifold $\left(\Omega,\{\cdot,\cdot\}_{\nu;C_1,\dots,C_{n-2}}\right)$. 

Let $(\Omega^{\prime},\Phi_1,\Phi_2)$ be a triple consisting of an open subset  $\Omega^{\prime}\subseteq \Omega$, and two analytic functions $\Phi_1,\Phi_2 \in\mathcal{C}^{\omega}(\Omega^{\prime},\mathbb{R})$, such that the map $\Phi:\Omega^{\prime}\rightarrow W^{\prime}=\Phi(\Omega^{\prime})$, given by $\Phi=(\Phi_1,\Phi_2,C_1,\dots,C_{n-2})$, is an analytic diffeomorphism. The choice of the open set $\Omega^{\prime}$, and respectively the analytic functions $\Phi_1,\Phi_2 \in\mathcal{C}^{\omega}(\Omega^{\prime},\mathbb{R})$, are made in accordance with the location of the periodic orbits (if any) of the completely integrable system \eqref{syscis1}. Let us fix a regular symplectic leaf, $\Sigma_{c_1 ,\dots, c_{n-2}}=C_{1}^{-1}(\{c_1 \})\cap\dots\cap C_{n-2}^{-1}(\{c_{n-2}\})$, of the Poisson manifold $\left(\Omega,\{\cdot,\cdot\}_{\nu;C_1,\dots,C_{n-2}}\right)$, and assume for the sake of simplicity that $\Sigma_{c_1 ,\dots, c_{n-2}}\cap\Omega^{\prime}$ is a connected subset of $\Sigma_{c_1 ,\dots, c_{n-2}}$ (if not, we shall consider instead a connected component of $\Sigma_{c_1 ,\dots, c_{n-2}}\cap\Omega^{\prime}$). Let us recall now the Hypothesis (H).

\medskip
\textbf{Hypothesis (H):} \textit{Suppose there exists a family of periodic orbits of the completely integrable system generated by the vector field \eqref{syscis1}, $\gamma_{h} \subseteq H|_{\Sigma_{c_1 ,\dots, c_{n-2}}\cap\Omega^{\prime}}^{-1}(\{h\})$ (which depends continuously on $h\in (a_{c_1,\dots,c_{n-2}},b_{c_1,\dots,c_{n-2}})$, for some $a_{c_1,\dots,c_{n-2}},b_{c_1,\dots,c_{n-2}}\in\mathbb{R}$), located on the open subset $\Sigma_{c_1 ,\dots, c_{n-2}}\cap\Omega^{\prime}$ of a regular symplectic leaf, $\Sigma_{c_1 ,\dots, c_{n-2}}=C_{1}^{-1}(\{c_1 \})\cap\dots\cap C_{n-2}^{-1}(\{c_{n-2}\})$, of the Poisson manifold $\left(\Omega,\{\cdot,\cdot\}_{\nu;C_1,\dots,C_{n-2}}\right)$.} 

\textit{Recall that the real numbers $a_{c_1 ,\dots, c_{n-2}}< b_{c_1 ,\dots, c_{n-2}}$ depend also on the open set $\Omega^{\prime}$; the reason we do not specify their dependence on $\Omega^{\prime}$ explicitly, is that $\Omega^{\prime}$ is supposed to be the largest subset of $\Omega$ where the Darboux normal form holds.}

\medskip
Then, by Theorem \eqref{nfcisp1} we obtain that $\Phi(\gamma_h)\subseteq (\Phi_{\star}H)|_{\Phi(\Sigma_{c_1 ,\dots, c_{n-2}}\cap\Omega^{\prime})}^{-1}(\{h\})$ (depending continuously on $h\in (a_{c_1,\dots,c_{n-2}},b_{c_1,\dots,c_{n-2}})$), are periodic orbits of the Hamilton-Poisson vector field $\Phi_{\star}X\in\mathfrak{X}(W^{\prime})$, located on the two-dimensional symplectic leaf $\{(y_1,\dots,y_n)\in W^{\prime}\mid y_3=c_1, \dots, y_n =c_{n-2}\}$ of the Poisson manifold $\left(W^{\prime},\{\cdot,\cdot\}_{\nu_{\Phi};\Phi_{\star}C_1,\dots,\Phi_{\star}C_{n-2}}\right)$. 

Moreover, recall also from Theorem \eqref{nfcisp1} that the restriction to the symplectic leaf $\Phi(\Sigma_{c_1,\dots,c_{n-2}}\cap \Omega^{\prime})=\{(y_1,\dots,y_n)\in W^{\prime}\mid y_3=c_1, \dots, y_n =c_{n-2}\}$ of the perturbed vector filed $\Phi_{\star}X_{\varepsilon}$ (where $X_{\varepsilon}=X+\varepsilon A$, with $A$ defined in Proposition \eqref{AformulaB}), has the expression
\begin{equation}\label{systreduced1}
\Phi_{\star}X_{\varepsilon}|_{\Phi(\Sigma_{c_1,\dots,c_{n-2}}\cap \Omega^{\prime})} = \nu_{\Phi}\cdot \Biggl[ \dfrac{\partial(\Phi_{\star}H)}{\partial y_2} + \varepsilon Q_1 \Biggr]\cdot\dfrac{\partial}{\partial y_1} + \nu_{\Phi}\cdot \Biggl[ - \dfrac{\partial(\Phi_{\star}H)}{\partial y_1} + \varepsilon Q_2 \Biggr]\cdot\dfrac{\partial}{\partial y_2}.
\end{equation}
Following \eqref{PPT}, for each $h\in(a_{c_1,\dots,c_{n-2}},b_{c_1,\dots,c_{n-2}})$, we define
\begin{equation*}
J_{c_1,\dots,c_{n-2}}(h):=\oint_{\Phi(\gamma_{h})}-Q_{1}(y_1,y_2,c_1,\dots,c_{n-2})\mathrm{d}y_2 + Q_{2}(y_1,y_2,c_1,\dots,c_{n-2})\mathrm{d}y_1.
\end{equation*}
Now one can state the main result of this section, which provides a criterion that guarantees the existence of bifurcations of limit cycles of a perturbed vector field $X_{\varepsilon}=X+\varepsilon A \in\mathfrak{X}(\Omega)$ (with $A|_{\Omega^{\prime}}$ defined as in Proposition \eqref{AformulaB}), located on the domain $\Sigma_{c_1 ,\dots, c_{n-2}}\cap\Omega^{\prime}$ of certain regular symplectic leaves, $\Sigma_{c_1 ,\dots, c_{n-2}}=C_{1}^{-1}(\{c_1 \})\cap\dots\cap C_{n-2}^{-1}(\{c_{n-2}\})$, of the Poisson configuration manifold $\left(\Omega,\{\cdot,\cdot\}_{\nu;C_1,\dots,C_{n-2}}\right)$ of the completely integrable system \eqref{syscis1}. Recall from Theorem \eqref{nfcisp1} that the orbits of the vector field $X_{\varepsilon}$, located on a domain $\Sigma_{c_1 ,\dots, c_{n-2}}\cap\Omega^{\prime}$, are in one-to-one correspondence with the orbits of the vector field $\Phi_{\star}X_{\varepsilon}$, located on the domain $\Phi(\Sigma_{c_1 ,\dots, c_{n-2}}\cap\Omega^{\prime})$. Before stating the main result, note that, if the Hypothesis (H) holds true on a regular symplectic leaf, then it will also hold true on nearby regular leaves.
\medskip

\begin{theorem}\label{MainThm2}
In the above settings, suppose there exists $(c_1^0,\dots,c_{n-2}^0)\in\mathbb{R}^{n-2}$ such that the Hypothesys (H) holds true on $\Sigma_{c_1^0,\dots,c_{n-2}^0}\cap \Omega^{\prime}$, and moreover, $J_{c_1^0,\dots,c_{n-2}^0}=J_{c_1^0,\dots,c_{n-2}^0}(h)$ is not identically zero on $(a_{c_1^0,\dots,c_{n-2}^0},b_{c_1^0,\dots,c_{n-2}^0})$. Let $U\subset\mathbb{R}^{n-2}$ be an open neighborhood of $(c_1^0,\dots,c_{n-2}^0)$, such the the Hypothesys (H) holds true on $\Sigma_{c_1,\dots,c_{n-2}}\cap \Omega^{\prime}$, for every $(c_1,\dots,c_{n-2})\in U$. Then, for every $(c_1,\dots,c_{n-2})\in U$ such that $J_{c_1,\dots,c_{n-2}}=J_{c_1,\dots,c_{n-2}}(h)$ is not identically zero on $(a_{c_1,\dots,c_{n-2}},b_{c_1,\dots,c_{n-2}})$, the following statements hold true.
\begin{itemize}
\item If the perturbed vector field $X_{\varepsilon}$ has a limit cycle on the domain $\Sigma_{c_1 ,\dots, c_{n-2}}\cap\Omega^{\prime}$ of the regular symplectic leaf $\Sigma_{c_1 ,\dots, c_{n-2}}$, bifurcating from the periodic orbit $\gamma_{h^{\star}}$ of the completely integrable vector field $X$, then $J_{c_1,\dots,c_{n-2}}(h^{\star})=0$.
\item If there exists $h^{\star}\in (a_{c_1 ,\dots, c_{n-2}},b_{c_1 ,\dots, c_{n-2}})$ a simple zero of $J_{c_1,\dots,c_{n-2}}$, then the perturbed vector field $X_{\varepsilon}$ has a unique limit cycle on the domain $\Sigma_{c_1 ,\dots, c_{n-2}}\cap\Omega^{\prime}$ of the regular symplectic leaf $\Sigma_{c_1 ,\dots, c_{n-2}}$, bifurcating from the periodic orbit $\gamma_{h^{\star}}$ of the completely integrable vector field $X$, and moreover, this limit cycle is hyperbolic.
\item If there exists $h^{\star}\in (a_{c_1 ,\dots, c_{n-2}},b_{c_1 ,\dots, c_{n-2}})$, a zero of order $k$ of $J_{c_1,\dots,c_{n-2}}$, then the perturbed vector field $X_{\varepsilon}$ has at most $k$ limit cycles (counting also the multiplicities) on the domain $\Sigma_{c_1 ,\dots, c_{n-2}}\cap\Omega^{\prime}$ of the regular symplectic leaf $\Sigma_{c_1 ,\dots, c_{n-2}}$, bifurcating from the periodic orbit $\gamma_{h^{\star}}$ of the completely integrable vector field $X$.
\item The total number (counting also the multiplicities) of limit cycles of the perturbed vector field $X_{\varepsilon}$, on the domain $\Sigma_{c_1 ,\dots, c_{n-2}}\cap\Omega^{\prime}$ of the regular symplectic leaf $\Sigma_{c_1 ,\dots, c_{n-2}}$, bifurcating from the annulus $\bigcup_{h\in(a,b)}\gamma_h$ of $X$, is bounded by the maximum number (if finite) of isolated zeros (counting also the multiplicities) of $J_{c_1,\dots,c_{n-2}}(h)$ for $h\in (a_{c_1 ,\dots, c_{n-2}},b_{c_1 ,\dots, c_{n-2}})$.
\end{itemize}
\end{theorem}
\begin{proof}
The proof follows from exactly the same type of arguments as the proof of Theorem \eqref{MainThm1}.
\end{proof}

\begin{remark}
The main difference between Theorem \eqref{MainThm2} and Theorem \eqref{MainThm1} originate in the construction of the perturbation vector field $A$ (given in, Proposition \eqref{AformulaB}, and respectively, Proposition \eqref{Aformula}), and consequently the difference between the associated perturbed vector fields, $X_{\varepsilon}=X+\varepsilon A$. More precisely, in the case of Theorem \eqref{MainThm1}, the perturbation vector filed $A$ is tangent \textbf{only} to the open domain $\Omega^{\prime}\cap\Sigma_{c_1,\dots,c_{n-2}}$ of the regular symplectic leaf $\Sigma_{c_1,\dots,c_{n-2}}$, while the perturbation vector filed used in Theorem \eqref{MainThm2} is tangent to \textbf{any regular symplectic leaf} of the Poisson manifold $(\Omega^{\prime},\{\cdot,\cdot\}_{\nu; C_1,\dots, C_{n-2}})$. 

To summarize, the Theorem \eqref{MainThm1} can be applied only for an a-priori fixed symplectic leaf, while the Theorem \eqref{MainThm2} works for all regular symplectic leaves, where the Hypothesis (H) holds true. 
\end{remark}

\section{Applications to Jacobi hyperelliptic functions}

The aim of this section is to apply the main theoretical results of this work, to the case of a concrete dynamical system. We shall consider as example, a real analytic completely integrable system which generates the so called \textit{Jacobi hyperelliptic functions} \cite{bates}. 

Let us recall form \cite{bates} the system which generates the Jacobi hyperelliptic functions. The system is a natural generalization of the one generating the classical Jacobi elliptic functions ($\operatorname{sn}$, $\operatorname{cn}$ and $\operatorname{dn}$), and is given by the following system of differential equations in $\mathbb{R}^n$ ($n\geq 3$):
\begin{equation}\label{hyperelli}
\left\{\begin{array}{l}
\dot x_{1}= \hat{x}_{1}x_2 \dots x_n\\
\dot x_{2}= -x_1 \hat{x}_{2} \dots x_n\\
\dot x_{3}= -k_{1}^{2} x_1 x_2 \hat{x}_3 \dots x_n\\
\cdots\\
\dot x_{n}= -k_{n-2}^{2} x_1 x_2 x_3 \dots \hat{x}_n,\\
\end{array}\right.
\end{equation}
where $\hat{x}$ means that the term $x$ is omitted, and $k_{1}, \dots , k_{n-2}$ are fixed nonzero real numbers.

Following \cite{bates}, the components of the solution of \eqref{hyperelli} with initial condition $x_1(0)=0$, $x_2(0)=\dots=x_n(0)=1$, are denoted by 
\begin{equation}\label{elli}
\left\{\begin{array}{l}
x_{1}(t)= \operatorname{sn}(t;k_1,\dots, k_{n-2}),\\
x_{2}(t)= \operatorname{cn}(t;k_1,\dots, k_{n-2}),\\
x_{3}(t)= \operatorname{dn}_{1}(t;k_1,\dots, k_{n-2}),\\
\cdots\\
x_{n}(t)= \operatorname{dn}_{n-2}(t;k_1,\dots, k_{n-2}),\\
\end{array}\right.
\end{equation}
and are called the \textit{Jacobi hyperelliptic functions}.

In order to illustrate the main theoretical results of this article, we need first to write the system \eqref{hyperelli} as a real analytic completely integrable system. This follows simply by recalling from \cite{bates} that $H=\dfrac{1}{2}(x_1^2 + x_2^2)$, $C_1=\dfrac{1}{2}(k_1^2 x_1^2 + x_3^2)$, $\dots$, $C_{n-2}=\dfrac{1}{2}(k_{n-2}^2 x_1^2 + x_n^2)$, are first integrals of \eqref{hyperelli}. Consequently, using the results presented at the beginning of the second section, some simple computations lead to the following result. 

\begin{proposition}
The system \eqref{hyperelli} can be realized as a Hamiltonian system of the type \eqref{systy}, where the Hamiltonian is $H(x_1,\dots,x_n)=\dfrac{1}{2}(x_1^2 + x_2^2)$, and the Poisson manifold is given by $(\mathbb{R}^{n},\{\cdot,\cdot\}_{\nu;C_1,\dots,C_{n-2}})$, where $C_1(x_1,\dots,x_n)=\dfrac{1}{2}(k_1^2 x_1^2 + x_3^2)$, $\dots$, $C_{n-2}(x_1,\dots,x_n)=\dfrac{1}{2}(k_{n-2}^2 x_1^2 + x_n^2)$, and $\nu(x_1,\dots,x_n)=1$.
\end{proposition}

As pointed out in \cite{bates}, a consequence of the complete integrability of \eqref{hyperelli}, leads to the motivation of the term \textit{hyperelliptic} used in the definition of the solution of the Cauchy problem generated by the equation \eqref{hyperelli} together with the initial condition $x_1(0)=0$, $x_2(0)=\dots=x_n(0)=1$. Indeed, if we square of the first equation of \eqref{hyperelli}, and we express $x_2^2,\dots,x_n^2$ in terms of $x_1^2$ (using the first integrals $H,C_1,\dots,C_{n-2}$, and the initial condition), one obtains that the first component of the solution of the Cauchy problem, inverts a hyperelliptic integral, i.e.,
$$
\int_{0}^{\operatorname{sn}(t;k_1,\dots,k_{n-2})}\dfrac{du}{\sqrt{(1-u^2)(1-k_1^2 u^2)\dots(1-k_{n-2}^2 u^2)}}=t.
$$ 

Let us return now to the main part of this section, namely, to illustrate the main results of this paper (i.e, Theorem \eqref{MainThm2} and Theorem \eqref{MainThm1}) in the case of the real analytic completely integrable system \eqref{hyperelli}. In order to do that, the second step after writing the system in the form  \eqref{systy}, is to show that the Hypothesis (H) holds true.

Using the same notations as in Hypothesis (H), for each $\varepsilon_3,\dots,\varepsilon_n\in\{-1,1\}$, we define $$\Omega^{\prime}_{\varepsilon_3,\dots,\varepsilon_n}:=\{(x_1,x_2,x_3,\dots,x_n)\in\mathbb{R}^n \mid \varepsilon_3 x_3>0, \dots, \varepsilon_n x_n>0\}\subset \Omega:=\mathbb{R}^n,$$ and $\Phi_1,\Phi_2 : \Omega^{\prime}_{\varepsilon_3,\dots,\varepsilon_n}\longrightarrow \mathbb{R}$, $\Phi_1(x_1,x_2,x_3,\dots,x_n):=x_1$, $\Phi_2(x_1,x_2,x_3,\dots,x_n):=x_2$.

A simple computation shows that the map $\Phi_{\varepsilon_3,\dots,\varepsilon_n}:\Omega^{\prime}_{\varepsilon_3,\dots,\varepsilon_n}\longrightarrow \Phi_{\varepsilon_3,\dots,\varepsilon_n}(\Omega^{\prime}_{\varepsilon_3,\dots,\varepsilon_n}):=W^{\prime}$, given by $\Phi_{\varepsilon_3,\dots,\varepsilon_n}:=(\Phi_1,\Phi_2,C_1,\dots,C_{n-2})$, is an analytic diffeomorphism between $\Omega^{\prime}_{\varepsilon_3,\dots,\varepsilon_n}$ and $W^{\prime}$, where
$$
W^{\prime}=\{(y_1,\dots,y_n)\in\mathbb{R}^n \mid k_1^2 y_1^2<2 y_3,\dots, k_{n-2}^2 y_1^2<2 y_n\}.
$$

If one denotes by 
$$
X:=\hat{x}_{1}x_2 \dots x_n \cdot \dfrac{\partial}{\partial x_1}-x_1 \hat{x}_{2} \dots x_n \cdot \dfrac{\partial}{\partial x_2}-k_{1}^{2} x_1 x_2 \hat{x}_3 \cdot\dfrac{\partial}{\partial x_3} - \dots  -k_{n-2}^{2} x_1 x_2 x_3 \dots \hat{x}_n \cdot \dfrac{\partial}{\partial x_n},
$$
the restriction to $\Omega^{\prime}_{\varepsilon_3,\dots,\varepsilon_n}$ of the vector field which generates the dynamical system \eqref{hyperelli}, then the vector field $(\Phi_{\varepsilon_3,\dots,\varepsilon_n})_{\star}X\in\mathfrak{X}(W^\prime)$ has the expression
$$
(\Phi_{\varepsilon_3,\dots,\varepsilon_n})_{\star}X=\varepsilon_3 \dots \varepsilon_n\cdot \sqrt{(2y_3-k_1^2y_1^2)\cdot\dots\cdot(2y_n-k_{n-2}^2y_1^2)}\cdot\left(y_2 \cdot \dfrac{\partial}{\partial y_1} - y_1 \cdot \dfrac{\partial}{\partial y_2}\right).
$$
Consequently, the system \eqref{hyperelli} written in coordinates $(y_1,\dots,y_n)\in W^{\prime}$ becomes
\begin{align}\label{hyperdarb}
\left\{\begin{array}{l}
\dot y_{1}= \varepsilon_3 \dots \varepsilon_n \sqrt{(2y_3-k_1^2y_1^2)\cdot\dots\cdot (2y_n-k_{n-2}^2y_1^2)}\cdot y_2\\
\dot y_{2}= \varepsilon_3 \dots \varepsilon_n \sqrt{(2y_3-k_1^2y_1^2)\cdot\dots\cdot (2y_n-k_{n-2}^2y_1^2)}\cdot (-y_1)\\
\dot y_{3}= 0\\
\cdots\\
\dot y_{n}= 0.\\
\end{array}\right.
\end{align}
Moreover, by Theorem \eqref{darbnf}, the system \eqref{hyperdarb} admits a Hamiltonian realization of the type \eqref{systy}, with Hamiltonian $$(\Phi_{\varepsilon_3,\dots,\varepsilon_n})_{\star}H(y_1,\dots,y_n)=\dfrac{1}{2}(y_1^2 + y_2^2),$$ modeled on the Poisson manifold $(W^{\prime},\{\cdot,\cdot\}_{\nu_{\Phi_{\varepsilon_3,\dots,\varepsilon_n}};(\Phi_{\varepsilon_3,\dots,\varepsilon_n})_{\star}C_1,\dots,(\Phi_{\varepsilon_3,\dots,\varepsilon_n})_{\star}C_{n-2}})$, where $$(\Phi_{\varepsilon_3,\dots,\varepsilon_n})_{\star}C_1(y_1,\dots,y_n)=y_3,\dots, (\Phi_{\varepsilon_3,\dots,\varepsilon_n})_{\star}C_{n-2}(y_1,\dots,y_n)=y_n,$$ and $\nu_{\Phi_{\varepsilon_3,\dots,\varepsilon_n}}(y_1,\dots,y_n)=\varepsilon_3 \dots \varepsilon_n \sqrt{(2y_3-k_1^2y_1^2)\cdot\dots\cdot(2y_n-k_{n-2}^2y_1^2)}$. More precisely, this follows since the analytic diffeomorphism $$\Phi_{\varepsilon_3,\dots,\varepsilon_n}:(\Omega^{\prime}_{\varepsilon_3,\dots,\varepsilon_n},\{\cdot,\cdot\}_{\nu;C_1,\dots,C_{n-2}})\longrightarrow (W^{\prime},\{\cdot,\cdot\}_{\nu_{\Phi_{\varepsilon_3,\dots,\varepsilon_n}};(\Phi_{\varepsilon_3,\dots,\varepsilon_n})_{\star}C_1,\dots,(\Phi_{\varepsilon_3,\dots,\varepsilon_n})_{\star}C_{n-2}})$$ is also a Poisson isomorphism.

Let $c_1,\dots,c_{n-2}>0$ be strictly positive. Then, the set 
\begin{align*}
&\Sigma_{c_1,\dots,c_{n-2}}\cap\Omega^{\prime}_{\varepsilon_3,\dots,\varepsilon_n}:=\{(x_1,\dots,x_n)\in\Omega^{\prime}_{\varepsilon_3,\dots,\varepsilon_n} \mid C_1(x_1,\dots,x_n)=c_1, \dots, C_{n-2}(x_1,\dots,x_n)=c_{n-2}\}\\
&=\left\{\left(x_1,x_2,\varepsilon_3 \sqrt{2c_1-k_1^2x_1^2},\dots,\varepsilon_n \sqrt{2c_{n-2}-k_{n-2}^2x_1^2}\right)\in\mathbb{R}^n \mid - \rho_{c_1,\dots,c_{n-2}}<x_1<\rho_{c_1,\dots,c_{n-2}} \right\},
\end{align*}
where $\rho_{c_1,\dots,c_{n-2}}:=\min_{i=\overline{1,{n-2}}}\dfrac{\sqrt{2c_i}}{k_i}$, is a regular symplectic leaf of the Poisson manifold $(\Omega^{\prime}_{\varepsilon_3,\dots,\varepsilon_n},\{\cdot,\cdot\}_{\nu;C_1,\dots,C_{n-2}})$.

Consequently, the image of $\Sigma_{c_1,\dots,c_{n-2}}\cap\Omega^{\prime}_{\varepsilon_3,\dots,\varepsilon_n}$ through the Poisson isomorphism $\Phi_{\varepsilon_3,\dots,\varepsilon_n}$,
\begin{align*}
&\Phi_{\varepsilon_3,\dots,\varepsilon_n}(\Sigma_{c_1,\dots,c_{n-2}}\cap\Omega^{\prime}_{\varepsilon_3,\dots,\varepsilon_n})=\{(y_1,y_2,c_1,\dots,c_{n-2})\in\mathbb{R}^{n} \mid - \rho_{c_1,\dots,c_{n-2}}<y_1 < \rho_{c_1,\dots,c_{n-2}}\},
\end{align*}
is a regular symplectic leaf of the Poisson manifold $$(W^{\prime},\{\cdot,\cdot\}_{\nu_{\Phi_{\varepsilon_3,\dots,\varepsilon_n}};(\Phi_{\varepsilon_3,\dots,\varepsilon_n})_{\star}C_1,\dots,(\Phi_{\varepsilon_3,\dots,\varepsilon_n})_{\star}C_{n-2}}).$$

Since the restriction of the Hamiltonian system \eqref{hyperdarb} to the regular symplectic leaf $\Phi_{\varepsilon_3,\dots,\varepsilon_n}(\Sigma_{c_1,\dots,c_{n-2}}\cap\Omega^{\prime}_{\varepsilon_3,\dots,\varepsilon_n})$, is given by the system
\begin{align*}
\left\{\begin{array}{l}
\dot y_{1}= \varepsilon_3 \dots \varepsilon_n \sqrt{(2c_1-k_1^2y_1^2)\cdot\dots\cdot (2c_{n-2}-k_{n-2}^2y_1^2)}\cdot y_2\\
\dot y_{2}= \varepsilon_3 \dots \varepsilon_n \sqrt{(2c_1-k_1^2y_1^2)\cdot\dots\cdot (2c_{n-2}-k_{n-2}^2y_1^2)}\cdot (-y_1),\\
\end{array}\right.
\end{align*}
which has a center located at $(y_1,y_2)=(0,0)$, we get the existence of a continuous family of periodic orbits, represented by the circles $\Phi_{\varepsilon_3,\dots,\varepsilon_n}(\gamma_h^{\varepsilon_3,\dots,\varepsilon_n})\subset \Phi_{\varepsilon_3,\dots,\varepsilon_n}(\Sigma_{c_1,\dots,c_{n-2}}\cap\Omega^{\prime}_{\varepsilon_3,\dots,\varepsilon_n})$
\begin{equation}\label{circlesphi}
\Phi_{\varepsilon_3,\dots,\varepsilon_n}(\gamma_h^{\varepsilon_3,\dots,\varepsilon_n}):=\{(y_1,y_2,c_1,\dots,c_{n-2})\in\mathbb{R}^{n} \mid \dfrac{1}{2}(y_1^2+y_2^2)=h\}, \quad 0<h<\dfrac{1}{2}\rho^2_{c_1,\dots,c_{n-2}}. 
\end{equation}
As the symplectic leaves are dynamically invariant, the periodic orbits from the family \eqref{circlesphi}, are also periodic orbits for the system \eqref{hyperdarb}.

Consequently, transporting $\Phi_{\varepsilon_3,\dots,\varepsilon_n}(\gamma_h^{\varepsilon_3,\dots,\varepsilon_n})$ to $\Omega^{\prime}_{\varepsilon_3,\dots,\varepsilon_n}$ through the analytic diffeomorphism $\Phi_{\varepsilon_3,\dots,\varepsilon_n}^{-1}$, we obtain that the continuous family of closed curves $\gamma_h^{\varepsilon_3,\dots,\varepsilon_n} \subset \Sigma_{c_1,\dots,c_{n-2}}\cap\Omega^{\prime}_{\varepsilon_3,\dots,\varepsilon_n}$, $0<h<\dfrac{1}{2}\rho^2_{c_1,\dots,c_{n-2}}$, given by
\begin{align}\label{circles}
\gamma_h^{\varepsilon_3,\dots,\varepsilon_n}:&=\left\{ \left(x_1,x_2,\varepsilon_3 \sqrt{2c_1-k_1^2x_1^2},\dots,\varepsilon_n \sqrt{2c_{n-2}-k_{n-2}^2x_1^2}\right)\in\mathbb{R}^n \mid \dfrac{1}{2}(x_1^2+x_2^2)=h\right\},
\end{align}
are periodic orbits of the system \eqref{hyperelli}. 

\textit{Hence, the Hypothesis (H) is verified on each regular symplectic leaf $\Sigma_{c_1,\dots,c_{n-2}}\cap\Omega^{\prime}_{\varepsilon_3,\dots,\varepsilon_n}$, ($c_1,\dots,c_{n-2}>0$), of the Poisson manifold $(\Omega^{\prime}_{\varepsilon_3,\dots,\varepsilon_n},\{\cdot,\cdot\}_{\nu;C_1,\dots,C_{n-2}})$.}

The Hypothesis (H) being verified, one can start analyzing the bifurcation problem. Recall that the general bifurcation problem we are dealing with, is mainly concerned with the persistence, and respectively upper bounds estimates on the number of periodic orbits that survive under analytic perturbations. More precisely, if one considers an analytic perturbation of the vector field $X$ which generates the integrable system \eqref{hyperelli}, $X_{\varepsilon}:=X+\varepsilon A \in \mathfrak{X}(\Omega^{\prime}_{\varepsilon_3,\dots,\varepsilon_n})$, ($\varepsilon>0$), we demand under which conditions there exist $h^{\star}\in (0,1/2 \rho^2_{c_1,\dots,c_{n-2}})$ and $\varepsilon^{\star}>0$, such that $X_\varepsilon$ has a limit cycle $\Gamma_{\varepsilon}^{\varepsilon_3,\dots,\varepsilon_n}\subset\Omega^{\prime}_{\varepsilon_3,\dots,\varepsilon_n}$, for each $0<\varepsilon<\varepsilon^{\star}$, and $\lim_{\varepsilon\rightarrow 0}\Gamma_{\varepsilon}^{\varepsilon_3,\dots,\varepsilon_n}=\gamma_{h^{\star}}^{\varepsilon_3,\dots,\varepsilon_n}$, in the sense of Hausdorff distance? Recall that, in this case we said the limit cycle $\Gamma_{\varepsilon}^{\varepsilon_3,\dots,\varepsilon_n}$ bifurcate from the periodic orbit $\gamma_{h^{\star}}^{\varepsilon_3,\dots,\varepsilon_n}$. Moreover, besides the existence problem, we are interested in providing an upper bound for the total number of limit cycles of $X_{\varepsilon}$ (located on $\Sigma_{c_1,\dots,c_{n-2}}\cap\Omega^{\prime}_{\varepsilon_3,\dots,\varepsilon_n}$) which bifurcate from the family $\bigcup_{h\in(0,1/2 \rho^2_{c_1,\dots,c_{n-2}})}\gamma_{h}^{\varepsilon_3,\dots,\varepsilon_n}$ of periodic orbits of the vector field $X$.

At this point one distinguish two cases: first, if the perturbed vector field $X_{\varepsilon}\in \mathfrak{X}(\Omega^{\prime}_{\varepsilon_3,\dots,\varepsilon_n})$ keeps dynamically invariant some fixed regular symplectic leaf $\Sigma_{c_1,\dots,c_{n-2}}\cap\Omega^{\prime}_{\varepsilon_3,\dots,\varepsilon_n}$ of the Poisson manifold $(\Omega^{\prime}_{\varepsilon_3,\dots,\varepsilon_n},\{\cdot,\cdot\}_{\nu;C_1,\dots,C_{n-2}})$, and second, if the perturbed vector field $X_{\varepsilon}$ keeps dynamically invariant the whole regular part of the symplectic foliation of the Poisson manifold $(\Omega^{\prime}_{\varepsilon_3,\dots,\varepsilon_n},\{\cdot,\cdot\}_{\nu;C_1,\dots,C_{n-2}})$. The main difference between the above approaches is the expression of the perturbation vector field $A$, as provided by Proposition \eqref{Aformula}, and Proposition \eqref{AformulaB}. Let us start analyzing the first case.
\medskip

\subsection{Perturbations keeping invariant an a-priori fixed symplectic leaf}
\medskip
The aim of this subsection is to give an answer to the limit cycles bifurcation problem, in the case of an analytic perturbation of the system \eqref{hyperelli} which keeps dynamically invariant an a-priory fixed regular symplectic leaf, $\Sigma_{c_1,\dots,c_{n-2}}\cap\Omega^{\prime}_{\varepsilon_3,\dots,\varepsilon_n}$, of the Poisson manifold $(\Omega^{\prime}_{\varepsilon_3,\dots,\varepsilon_n},\{\cdot,\cdot\}_{\nu;C_1,\dots,C_{n-2}})$.

Using the Proposition \eqref{Aformula}, some straightforward computations lead to the following general form of the perturbation vector filed which keeps invariant an a-priory fixed regular symplectic leaf, $\Sigma_{c_1,\dots,c_{n-2}}\cap\Omega^{\prime}_{\varepsilon_3,\dots,\varepsilon_n}$:
\begin{align*}
A &|_{\Omega^{\prime}_{\varepsilon_3,\dots,\varepsilon_n}}=\hat{x}_{1}\hat{x}_{2}x_3\dots x_n \cdot P_1(x_1,x_2,\dfrac{1}{2}(k_1^2 x_1^2+x_3^2),\dots,\dfrac{1}{2}(k_{n-2}^2 x_1^2+x_n^2))\cdot \dfrac{\partial}{\partial x_1}\\
&+ \hat{x}_{1}\hat{x}_{2}x_3\dots x_n \cdot P_2(x_1,x_2,\dfrac{1}{2}(k_1^2 x_1^2+x_3^2),\dots,\dfrac{1}{2}(k_{n-2}^2 x_1^2+x_n^2))\cdot \dfrac{\partial}{\partial x_2}\\
&+[ -k_1^2 x_1\hat{x}_{2}\hat{x}_3\dots x_n \cdot P_1(x_1,x_2,\dfrac{1}{2}(k_1^2 x_1^2+x_3^2),\dots,\dfrac{1}{2}(k_{n-2}^2 x_1^2+x_n^2)) \\
&+ \sum_{i=1}^{n-2}\dfrac{k_i^2 x_1^2+x_{i+2}^2 - 2 c_i}{2 x_3}\cdot R_{1,i}(x_1,x_2,\dfrac{1}{2}(k_1^2 x_1^2+x_3^2),\dots,\dfrac{1}{2}(k_{n-2}^2 x_1^2+x_n^2))]\cdot \dfrac{\partial}{\partial x_3}\\
&+\dots +\\
&+[ -k_{n-2}^2 x_1\hat{x}_{2}x_3\dots \hat{x}_n \cdot P_1(x_1,x_2,\dfrac{1}{2}(k_1^2 x_1^2+x_3^2),\dots,\dfrac{1}{2}(k_{n-2}^2 x_1^2+x_n^2)) \\
&+ \sum_{i=1}^{n-2}\dfrac{k_i^2 x_1^2+x_{i+2}^2 - 2 c_i}{2 x_n}\cdot R_{n-2,i}(x_1,x_2,\dfrac{1}{2}(k_1^2 x_1^2+x_3^2),\dots,\dfrac{1}{2}(k_{n-2}^2 x_1^2+x_n^2))]\cdot \dfrac{\partial}{\partial x_n},
\end{align*}
where $P_1,P_2,R_{i,j}\in\mathcal{C}^{\omega}(W^{\prime},\mathbb{R})$, $i,j\in\{1,\dots,{n-2}\}$, are arbitrary real analytic functions defined on 
$$
W^{\prime}=\Phi_{\varepsilon_3,\dots,\varepsilon_n}(\Omega^{\prime}_{\varepsilon_3,\dots,\varepsilon_n})=\{(y_1,\dots,y_n)\in\mathbb{R}^n \mid k_1^2 y_1^2<2 y_3,\dots, k_{n-2}^2 y_1^2<2 y_n\}.
$$

Consequently, the general analytic perturbation of the system \eqref{hyperelli} which keeps dynamically invariant the a-priori fixed regular symplectic leaf  $\Sigma_{c_1,\dots,c_{n-2}}\cap\Omega^{\prime}_{\varepsilon_3,\dots,\varepsilon_n}$, is given by:

\begin{align}\label{hyperpert1}
\left\{\begin{array}{l}
\dot x_{1}= \hat{x}_{1}x_2 \dots x_n + \varepsilon \cdot \hat{x}_{1}\hat{x}_{2}x_3\dots x_n \cdot P_1(x_1,x_2,\dfrac{1}{2}(k_1^2 x_1^2+x_3^2),\dots,\dfrac{1}{2}(k_{n-2}^2 x_1^2+x_n^2))\\
\dot x_{2}= -x_1 \hat{x}_{2} \dots x_n + \varepsilon \cdot \hat{x}_{1}\hat{x}_{2}x_3\dots x_n \cdot P_2(x_1,x_2,\dfrac{1}{2}(k_1^2 x_1^2+x_3^2),\dots,\dfrac{1}{2}(k_{n-2}^2 x_1^2+x_n^2))\\
\dot x_{3}= -k_{1}^{2} x_1 x_2 \hat{x}_3 \dots x_n + \varepsilon \cdot [ -k_1^2 x_1\hat{x}_{2}\hat{x}_3\dots x_n \cdot P_1(x_1,x_2,\dfrac{1}{2}(k_1^2 x_1^2+x_3^2),\dots,\dfrac{1}{2}(k_{n-2}^2 x_1^2+x_n^2)) \\
+ \sum_{i=1}^{n-2}\dfrac{k_i^2 x_1^2+x_{i+2}^2 - 2 c_i}{2 x_3}\cdot R_{1,i}(x_1,x_2,\dfrac{1}{2}(k_1^2 x_1^2+x_3^2),\dots,\dfrac{1}{2}(k_{n-2}^2 x_1^2+x_n^2))]\\
\cdots\\
\dot x_{n}= -k_{n-2}^{2} x_1 x_2 x_3 \dots \hat{x}_n + \varepsilon \cdot [ -k_{n-2}^2 x_1\hat{x}_{2}x_3\dots \hat{x}_n \cdot P_1(x_1,x_2,\dfrac{1}{2}(k_1^2 x_1^2+x_3^2),\dots,\dfrac{1}{2}(k_{n-2}^2 x_1^2+x_n^2)) \\
+ \sum_{i=1}^{n-2}\dfrac{k_i^2 x_1^2+x_{i+2}^2 - 2 c_i}{2 x_n}\cdot R_{n-2,i}(x_1,x_2,\dfrac{1}{2}(k_1^2 x_1^2+x_3^2),\dots,\dfrac{1}{2}(k_{n-2}^2 x_1^2+x_n^2))],\\
\end{array}\right.
\end{align}
where $P_1,P_2,R_{i,j}\in\mathcal{C}^{\omega}(W^{\prime},\mathbb{R})$, $i,j\in\{1,\dots,{n-2}\}$, are arbitrary real analytic functions.

Using the analytic diffeomorphism $\Phi_{\varepsilon_3,\dots,\varepsilon_n}:\Omega^{\prime}_{\varepsilon_3,\dots,\varepsilon_n}\longrightarrow W^{\prime}$,
$$
\Phi_{\varepsilon_3,\dots,\varepsilon_n}(x_1,\dots,x_n)=\left(x_1,x_2,\dfrac{1}{2}(k_1^2 x_1^2+x_3^2),\dots,\dfrac{1}{2}(k_{n-2}^2 x_1^2+x_n^2)\right),
$$
the perturbed system \eqref{hyperpert1} written in coordinates $(y_1,\dots,y_n)\in W^{\prime}$, becomes
\begin{align}\label{hyperpert1red}
\left\{\begin{array}{l}
\dot y_{1}= \varepsilon_3 \dots \varepsilon_n \sqrt{(2 y_3 -k_1^2 y_1^2)\dots (2 y_n -k_{n-2}^2 y_1^2)}\cdot (y_2 + \varepsilon P_1(y_1,\dots, y_n))\\
\dot y_{2}= \varepsilon_3 \dots \varepsilon_n \sqrt{(2 y_3 -k_1^2 y_1^2)\dots (2 y_n -k_{n-2}^2 y_1^2)}\cdot (-y_1 + \varepsilon P_2(y_1,\dots, y_n))\\
\dot y_{3}= \varepsilon \cdot \sum_{i=1}^{n-2}(y_{i+2}-c_i)R_{1,i}(y_1,\dots, y_n)\\
\cdots\\
\dot y_{n}= \varepsilon \cdot \sum_{i=1}^{n-2}(y_{i+2}-c_i)R_{n-2,i}(y_1,\dots, y_n).\\
\end{array}\right.
\end{align}
Moreover, since $$\Phi_{\varepsilon_3,\dots,\varepsilon_n}:(\Omega^{\prime}_{\varepsilon_3,\dots,\varepsilon_n},\{\cdot,\cdot\}_{\nu;C_1,\dots,C_{n-2}})\longrightarrow (W^{\prime},\{\cdot,\cdot\}_{\nu_{\Phi_{\varepsilon_3,\dots,\varepsilon_n}};(\Phi_{\varepsilon_3,\dots,\varepsilon_n})_{\star}C_1,\dots,(\Phi_{\varepsilon_3,\dots,\varepsilon_n})_{\star}C_{n-2}})$$ is a Poisson isomorphism, the dynamically invariant set of \eqref{hyperpert1red},
\begin{align*}
&\Phi_{\varepsilon_3,\dots,\varepsilon_n}(\Sigma_{c_1,\dots,c_{n-2}}\cap\Omega^{\prime}_{\varepsilon_3,\dots,\varepsilon_n})=\{(y_1,y_2,c_1,\dots,c_{n-2})\in\mathbb{R}^{n} \mid - \rho_{c_1,\dots,c_{n-2}}<y_1 < \rho_{c_1,\dots,c_{n-2}}\},
\end{align*}
is a regular symplectic leaf of the Poisson manifold $$(W^{\prime},\{\cdot,\cdot\}_{\nu_{\Phi_{\varepsilon_3,\dots,\varepsilon_n}};(\Phi_{\varepsilon_3,\dots,\varepsilon_n})_{\star}C_1,\dots,(\Phi_{\varepsilon_3,\dots,\varepsilon_n})_{\star}C_{n-2}}).$$

The restriction of \eqref{hyperpert1red} to the dynamically invariant set $\Phi(\Sigma_{c_1,\dots,c_{n-2}}\cap\Omega^{\prime}_{\varepsilon_3,\dots,\varepsilon_n})$ becomes 
\begin{equation}\label{erte}
\left\{\begin{array}{l}
\dot y_{1}= \varepsilon_3 \dots \varepsilon_n \sqrt{(2 c_1 -k_1^2 y_1^2)\dots (2 c_{n-2} -k_{n-2}^2 y_1^2)}\cdot (y_2 + \varepsilon P_1(y_1,y_2,c_1,\dots, c_{n-2}))\\
\dot y_{2}= \varepsilon_3 \dots \varepsilon_n \sqrt{(2 c_1 -k_1^2 y_1^2)\dots (2 c_{n-2} -k_{n-2}^2 y_1^2)}\cdot (-y_1 + \varepsilon P_2(y_1,y_2,c_1,\dots, c_{n-2}))\\
y_{3}= c_1\\
\cdots\\
y_{n}= c_{n-2}.\\
\end{array}\right.
\end{equation}

Let us define for each $h\in(0,\dfrac{1}{2}\rho_{c_1,\dots,c_{n-2}}^2)$
$$
I_{c_1,\dots,c_{n-2}}(h):=\oint_{\{(y_1,y_2) | y_1^2+y_2^2=2h\}}-P_{1}(y_1,y_2,c_1,\dots,c_{n-2})\mathrm{d}y_2 + P_{2}(y_1,y_2,c_1,\dots,c_{n-2})\mathrm{d}y_1.
$$

Note that $I_{c_1,\dots,c_{n-2}}$ has the same expression for each dynamically invariant set\\
$\Phi(\Sigma_{c_1,\dots,c_{n-2}}\cap\Omega^{\prime}_{\varepsilon_3,\dots,\varepsilon_n})$, independently of the choice of $\varepsilon_3,\dots,\varepsilon_n\in\{-1,1\}$. 

Let us state now the main result of this subsection, namely the Theorem \eqref{MainThm1} in the case of the system \eqref{hyperpert1}.

\begin{theorem}\label{mainThmExample1}
Suppose that $I_{c_1,\dots,c_{n-2}}=I_{c_1,\dots,c_{n-2}}(h)$ is not identically zero on $(0,\dfrac{1}{2}\rho_{c_1,\dots,c_{n-2}}^2)$. Then the following statements hold true for every $\varepsilon_3,\dots,\varepsilon_n\in\{-1,1\}$.
\begin{itemize}
\item If the perturbed system \eqref{hyperpert1} has a limit cycle on the symplectic leaf $\Sigma_{c_1 ,\dots, c_{n-2}}\cap\Omega^{\prime}_{\varepsilon_3,\dots,\varepsilon_n}$, bifurcating from the periodic orbit $\gamma_{h^{\star}}^{\varepsilon_3,\dots,\varepsilon_n}$ of the completely integrable system \eqref{hyperelli}, then $I_{c_1,\dots,c_{n-2}}(h^{\star})=0$.
\item If there exists $h^{\star}\in (0,\dfrac{1}{2}\rho_{c_1,\dots,c_{n-2}}^2)$ a simple zero of $I_{c_1,\dots,c_{n-2}}$, then the perturbed system \eqref{hyperpert1} has a unique limit cycle on the symplectic leaf $\Sigma_{c_1 ,\dots, c_{n-2}}\cap\Omega^{\prime}_{\varepsilon_3,\dots,\varepsilon_n}$, bifurcating from the periodic orbit $\gamma_{h^{\star}}^{\varepsilon_3,\dots,\varepsilon_n}$ of the completely integrable system \eqref{hyperelli}, and moreover, this limit cycle is hyperbolic.
\item If there exists $h^{\star}\in (0,\dfrac{1}{2}\rho_{c_1,\dots,c_{n-2}}^2)$, a zero of order $k$ of $I_{c_1,\dots,c_{n-2}}$, then the perturbed system \eqref{hyperpert1} has at most $k$ limit cycles (counting also the multiplicities) on the symplectic leaf $\Sigma_{c_1 ,\dots, c_{n-2}}\cap\Omega^{\prime}_{\varepsilon_3,\dots,\varepsilon_n}$, bifurcating from the periodic orbit $\gamma_{h^{\star}}^{\varepsilon_3,\dots,\varepsilon_n}$ of the completely integrable \eqref{hyperelli}.
\item The total number (counting also the multiplicities) of limit cycles of the perturbed system \eqref{hyperpert1}, on the symplectic leaf $\Sigma_{c_1 ,\dots, c_{n-2}}\cap\Omega^{\prime}_{\varepsilon_3,\dots,\varepsilon_n}$, bifurcating from the family $\bigcup_{h\in (0,1/2\rho_{c_1,\dots,c_{n-2}}^2)}\gamma_h^{\varepsilon_3,\dots,\varepsilon_n}$ of periodic orbits of the completely integrable \eqref{hyperelli}, is bounded by the maximum number (if finite) of isolated zeros (counting also the multiplicities) of $I_{c_1,\dots,c_{n-2}}(h)$ for $h\in (0,\dfrac{1}{2}\rho_{c_1,\dots,c_{n-2}}^2)$.
\end{itemize}
\end{theorem}

\subsection{Perturbations keeping invariant the regular part of the symplectic foliation}
\medskip

The aim of this subsection is to give an answer to the limit cycles bifurcation problem, in the case of an analytic perturbation of the system \eqref{hyperelli} which keeps dynamically invariant the regular part of the symplectic foliation of the of the Poisson manifold $(\Omega^{\prime}_{\varepsilon_3,\dots,\varepsilon_n},\{\cdot,\cdot\}_{\nu;C_1,\dots,C_{n-2}})$.

Using the Proposition \eqref{AformulaB}, some straightforward computations lead to the following general form of the perturbation vector filed which keeps invariant the regular part of the symplectic foliation of the of $(\Omega^{\prime}_{\varepsilon_3,\dots,\varepsilon_n},\{\cdot,\cdot\}_{\nu;C_1,\dots,C_{n-2}})$:
\begin{align*}
A |_{\Omega^{\prime}_{\varepsilon_3,\dots,\varepsilon_n}}&=\hat{x}_{1}\hat{x}_{2}x_3\dots x_n \cdot Q_1(x_1,x_2,\dfrac{1}{2}(k_1^2 x_1^2+x_3^2),\dots,\dfrac{1}{2}(k_{n-2}^2 x_1^2+x_n^2))\cdot \dfrac{\partial}{\partial x_1}\\
&+ \hat{x}_{1}\hat{x}_{2}x_3\dots x_n \cdot Q_2(x_1,x_2,\dfrac{1}{2}(k_1^2 x_1^2+x_3^2),\dots,\dfrac{1}{2}(k_{n-2}^2 x_1^2+x_n^2))\cdot \dfrac{\partial}{\partial x_2}\\
&-k_1^2 x_1\hat{x}_{2}\hat{x}_3\dots x_n \cdot Q_1(x_1,x_2,\dfrac{1}{2}(k_1^2 x_1^2+x_3^2),\dots,\dfrac{1}{2}(k_{n-2}^2 x_1^2+x_n^2)) \cdot \dfrac{\partial}{\partial x_3}\\
&-\dots -\\
& -k_{n-2}^2 x_1\hat{x}_{2}x_3\dots \hat{x}_n \cdot Q_1(x_1,x_2,\dfrac{1}{2}(k_1^2 x_1^2+x_3^2),\dots,\dfrac{1}{2}(k_{n-2}^2 x_1^2+x_n^2)) \cdot \dfrac{\partial}{\partial x_n},
\end{align*}
where $Q_1,Q_2\in\mathcal{C}^{\omega}(W^{\prime},\mathbb{R})$, are arbitrary real analytic functions defined on 
$$
W^{\prime}=\Phi_{\varepsilon_3,\dots,\varepsilon_n}(\Omega^{\prime}_{\varepsilon_3,\dots,\varepsilon_n})=\{(y_1,\dots,y_n)\in\mathbb{R}^n \mid k_1^2 y_1^2<2 y_3,\dots, k_{n-2}^2 y_1^2<2 y_n\}.
$$

Consequently, the general analytic perturbation of the system \eqref{hyperelli} which keeps dynamically invariant the regular part of the symplectic foliation of the of the Poisson manifold $(\Omega^{\prime}_{\varepsilon_3,\dots,\varepsilon_n},\{\cdot,\cdot\}_{\nu;C_1,\dots,C_{n-2}})$, is given by:

\begin{align}\label{hyperpert2}
\left\{\begin{array}{l}
\dot x_{1}= \hat{x}_{1}x_2 \dots x_n + \varepsilon \cdot \hat{x}_{1}\hat{x}_{2}x_3\dots x_n \cdot Q_1(x_1,x_2,\dfrac{1}{2}(k_1^2 x_1^2+x_3^2),\dots,\dfrac{1}{2}(k_{n-2}^2 x_1^2+x_n^2))\\
\dot x_{2}= -x_1 \hat{x}_{2} \dots x_n + \varepsilon \cdot \hat{x}_{1}\hat{x}_{2}x_3\dots x_n \cdot Q_2(x_1,x_2,\dfrac{1}{2}(k_1^2 x_1^2+x_3^2),\dots,\dfrac{1}{2}(k_{n-2}^2 x_1^2+x_n^2))\\
\dot x_{3}= -k_{1}^{2} x_1 x_2 \hat{x}_3 \dots x_n - \varepsilon \cdot k_1^2 x_1\hat{x}_{2}\hat{x}_3\dots x_n \cdot Q_1(x_1,x_2,\dfrac{1}{2}(k_1^2 x_1^2+x_3^2),\dots,\dfrac{1}{2}(k_{n-2}^2 x_1^2+x_n^2))\\
\cdots\\
\dot x_{n}= -k_{n-2}^{2} x_1 x_2 x_3 \dots \hat{x}_n - \varepsilon \cdot k_{n-2}^2 x_1\hat{x}_{2}x_3\dots \hat{x}_n \cdot Q_1(x_1,x_2,\dfrac{1}{2}(k_1^2 x_1^2+x_3^2),\dots,\dfrac{1}{2}(k_{n-2}^2 x_1^2+x_n^2)),\\
\end{array}\right.
\end{align}
where $Q_1,Q_2 \in\mathcal{C}^{\omega}(W^{\prime},\mathbb{R})$, are arbitrary real analytic functions.

Using the analytic diffeomorphism $\Phi_{\varepsilon_3,\dots,\varepsilon_n}:\Omega^{\prime}_{\varepsilon_3,\dots,\varepsilon_n}\longrightarrow W^{\prime}$,
$$
\Phi_{\varepsilon_3,\dots,\varepsilon_n}(x_1,\dots,x_n)=\left(x_1,x_2,\dfrac{1}{2}(k_1^2 x_1^2+x_3^2),\dots,\dfrac{1}{2}(k_{n-2}^2 x_1^2+x_n^2)\right),
$$
the perturbed system \eqref{hyperpert2} written in coordinates $(y_1,\dots,y_n)\in W^{\prime}$, becomes
\begin{align}\label{hyperpert2red}
\left\{\begin{array}{l}
\dot y_{1}= \varepsilon_3 \dots \varepsilon_n \sqrt{(2 y_3 -k_1^2 y_1^2)\dots (2 y_n -k_{n-2}^2 y_1^2)}\cdot (y_2 + \varepsilon Q_1(y_1,\dots, y_n))\\
\dot y_{2}= \varepsilon_3 \dots \varepsilon_n \sqrt{(2 y_3 -k_1^2 y_1^2)\dots (2 y_n -k_{n-2}^2 y_1^2)}\cdot (-y_1 + \varepsilon Q_2(y_1,\dots, y_n))\\
\dot y_{3}= 0\\
\cdots\\
\dot y_{n}= 0.\\
\end{array}\right.
\end{align}
Moreover, since $$\Phi_{\varepsilon_3,\dots,\varepsilon_n}:(\Omega^{\prime}_{\varepsilon_3,\dots,\varepsilon_n},\{\cdot,\cdot\}_{\nu;C_1,\dots,C_{n-2}})\longrightarrow (W^{\prime},\{\cdot,\cdot\}_{\nu_{\Phi_{\varepsilon_3,\dots,\varepsilon_n}};(\Phi_{\varepsilon_3,\dots,\varepsilon_n})_{\star}C_1,\dots,(\Phi_{\varepsilon_3,\dots,\varepsilon_n})_{\star}C_{n-2}})$$ is a Poisson isomorphism, then \textbf{for each} $c_1,\dots,c_{n-2}>0$, the dynamically invariant set of \eqref{hyperpert2red},
\begin{align*}
&\Phi_{\varepsilon_3,\dots,\varepsilon_n}(\Sigma_{c_1,\dots,c_{n-2}}\cap\Omega^{\prime}_{\varepsilon_3,\dots,\varepsilon_n})=\{(y_1,y_2,c_1,\dots,c_{n-2})\in\mathbb{R}^{n} \mid - \rho_{c_1,\dots,c_{n-2}}<y_1 < \rho_{c_1,\dots,c_{n-2}}\},
\end{align*}
is a regular symplectic leaf of the Poisson manifold $$(W^{\prime},\{\cdot,\cdot\}_{\nu_{\Phi_{\varepsilon_3,\dots,\varepsilon_n}};(\Phi_{\varepsilon_3,\dots,\varepsilon_n})_{\star}C_1,\dots,(\Phi_{\varepsilon_3,\dots,\varepsilon_n})_{\star}C_{n-2}}).$$

The restriction of \eqref{hyperpert2red} to the dynamically invariant set $\Phi(\Sigma_{c_1,\dots,c_{n-2}}\cap\Omega^{\prime}_{\varepsilon_3,\dots,\varepsilon_n})$ becomes 
\begin{equation}\label{erte2}
\left\{\begin{array}{l}
\dot y_{1}= \varepsilon_3 \dots \varepsilon_n \sqrt{(2 c_1 -k_1^2 y_1^2)\dots (2 c_{n-2} -k_{n-2}^2 y_1^2)}\cdot (y_2 + \varepsilon Q_1(y_1,y_2,c_1,\dots, c_{n-2}))\\
\dot y_{2}= \varepsilon_3 \dots \varepsilon_n \sqrt{(2 c_1 -k_1^2 y_1^2)\dots (2 c_{n-2} -k_{n-2}^2 y_1^2)}\cdot (-y_1 + \varepsilon Q_2(y_1,y_2,c_1,\dots, c_{n-2}))\\
y_{3}= c_1\\
\cdots\\
y_{n}= c_{n-2}.\\
\end{array}\right.
\end{equation}

Let us define for each $h\in(0,\dfrac{1}{2}\rho_{c_1,\dots,c_{n-2}}^2)$
$$
J_{c_1,\dots,c_{n-2}}(h):=\oint_{\{(y_1,y_2) | y_1^2+y_2^2=2h\}}-Q_{1}(y_1,y_2,c_1,\dots,c_{n-2})\mathrm{d}y_2 + Q_{2}(y_1,y_2,c_1,\dots,c_{n-2})\mathrm{d}y_1.
$$

Note that $J_{c_1,\dots,c_{n-2}}$ has the same expression for each dynamically invariant set\\
$\Phi(\Sigma_{c_1,\dots,c_{n-2}}\cap\Omega^{\prime}_{\varepsilon_3,\dots,\varepsilon_n})$, independently of the choice of $\varepsilon_3,\dots,\varepsilon_n\in\{-1,1\}$. 

Let us state now the main result of this subsection, namely the Theorem \eqref{MainThm2} in the case of the system \eqref{hyperpert2}.

\begin{theorem}\label{mainThmExample2}
Suppose there exist $c_1^0,\dots,c_{n-2}^0>0$ such that the map $J_{c_1^0,\dots,c_{n-2}^0}=J_{c_1^0,\dots,c_{n-2}^0}(h)$, is not identically zero on $(0,\dfrac{1}{2}\rho_{c_1^0,\dots,c_{n-2}^0}^2)$. Then the following statements hold true for every $\varepsilon_3,\dots,\varepsilon_n\in\{-1,1\}$, and respectively for every $c_1,\dots,c_{n-2}>0$ such that the associated map, $J_{c_1\dots,c_{n-2}}=J_{c_1\dots,c_{n-2}}(h)$, is not identically zero on $(0,\dfrac{1}{2}\rho_{c_1,\dots,c_{n-2}}^2)$.
\begin{itemize}
\item If the perturbed system \eqref{hyperpert2} has a limit cycle on the symplectic leaf $\Sigma_{c_1 ,\dots, c_{n-2}}\cap\Omega^{\prime}_{\varepsilon_3,\dots,\varepsilon_n}$, bifurcating from the periodic orbit $\gamma_{h^{\star}}^{\varepsilon_3,\dots,\varepsilon_n}$ of the completely integrable system \eqref{hyperelli}, then $J_{c_1,\dots,c_{n-2}}(h^{\star})=0$.
\item If there exists $h^{\star}\in (0,\dfrac{1}{2}\rho_{c_1,\dots,c_{n-2}}^2)$ a simple zero of $J_{c_1,\dots,c_{n-2}}$, then the perturbed system \eqref{hyperpert2} has a unique limit cycle on the symplectic leaf $\Sigma_{c_1 ,\dots, c_{n-2}}\cap\Omega^{\prime}_{\varepsilon_3,\dots,\varepsilon_n}$, bifurcating from the periodic orbit $\gamma_{h^{\star}}^{\varepsilon_3,\dots,\varepsilon_n}$ of the completely integrable system \eqref{hyperelli}, and moreover, this limit cycle is hyperbolic.
\item If there exists $h^{\star}\in (0,\dfrac{1}{2}\rho_{c_1,\dots,c_{n-2}}^2)$, a zero of order $k$ of $J_{c_1,\dots,c_{n-2}}$, then the perturbed system \eqref{hyperpert2} has at most $k$ limit cycles (counting also the multiplicities) on the symplectic leaf $\Sigma_{c_1 ,\dots, c_{n-2}}\cap\Omega^{\prime}_{\varepsilon_3,\dots,\varepsilon_n}$, bifurcating from the periodic orbit $\gamma_{h^{\star}}^{\varepsilon_3,\dots,\varepsilon_n}$ of the completely integrable \eqref{hyperelli}.
\item The total number (counting also the multiplicities) of limit cycles of the perturbed system \eqref{hyperpert2}, on the symplectic leaf $\Sigma_{c_1 ,\dots, c_{n-2}}\cap\Omega^{\prime}_{\varepsilon_3,\dots,\varepsilon_n}$, bifurcating from the family $\bigcup_{h\in (0,1/2\rho_{c_1,\dots,c_{n-2}}^2)}\gamma_h^{\varepsilon_3,\dots,\varepsilon_n}$ of periodic orbits of the completely integrable \eqref{hyperelli}, is bounded by the maximum number (if finite) of isolated zeros (counting also the multiplicities) of $J_{c_1,\dots,c_{n-2}}(h)$ for $h\in (0,\dfrac{1}{2}\rho_{c_1,\dots,c_{n-2}}^2)$.
\end{itemize}
\end{theorem}

\subsection{An example of sharp upper bound estimate for the number of leafwise bifurcating limit cycles}
\medskip
In this subsection we consider a special class of perturbations, such that the analytic functions $P_1,P_2$ from \eqref{hyperpert1}, and respectively, $Q_1,Q_2$ from \eqref{hyperpert2}, are given by
\begin{equation}\label{poly}
\begin{aligned}
P_1(y_1,\dots,y_n)=Q_1(y_1,\dots,y_n):=\sum_{0\leq i+j \leq m}r_{i,j}(y_3,\dots, y_n)y_1^i y_2^j,\\
P_2(y_1,\dots,y_n)=Q_2(y_1,\dots,y_n):=\sum_{0\leq i+j \leq m}s_{i,j}(y_3,\dots, y_n)y_1^i y_2^j,
\end{aligned}
\end{equation}
where $m\in\mathbb{N}$, $m\geq 1$, and $r_{i,j},s_{i,j}$ are real analytic functions on $\mathbb{R}^{n-2}$, for each $i,j\in\mathbb{N}$, with $0\leq i+j \leq m$. Suppose that both functions, $P_1=Q_1$, $P_2=Q_2$, are polynomial of degree $m$ in the first two coordinates.

Since the aim of this subsection is to provide a sharp upper bound estimate for the number of leafwise bifurcating limit cycles of the perturbed dynamical system \eqref{hyperpert1}, (and respectively \eqref{hyperpert2}), our approach is to apply the last item of Theorem \eqref{mainThmExample1} (and respectively Theorem \eqref{mainThmExample2}). In order to do that, note that for each $c_1,\dots,c_{n-2}>0$, the mappings $I_{c_1,\dots,c_{n-2}}$, $J_{c_1,\dots,c_{n-2}}$, are equal on $(0,1/2 \rho_{c_1,\dots,c_{n-2}}^2)$, where $\rho_{c_1,\dots,c_{n-2}}:=\min_{i=\overline{1,{n-2}}} \sqrt{2c_i} /{k_i}$. More precisely, for every $h\in (0,1/2 \rho_{c_1,\dots,c_{n-2}}^2)$ we have that
\begin{align*}
I_{c_1,\dots,c_{n-2}}(h)&=J_{c_1,\dots,c_{n-2}}(h)\\
&=\oint_{\{(y_1,y_2) | y_1^2+y_2^2=2h\}}-Q_{1}(y_1,y_2,c_1,\dots,c_{n-2})\mathrm{d}y_2 + Q_{2}(y_1,y_2,c_1,\dots,c_{n-2})\mathrm{d}y_1.
\end{align*}
Using the standard parameterization, $y_1=\sqrt{2h}\cos\theta$, $y_2=\sqrt{2h}\sin\theta$, $\theta\in[0,2\pi)$, after some straightforward computations, the integral $I_{c_1,\dots,c_{n-2}}(h)$ becomes
\begin{align*}
I_{c_1,\dots,c_{n-2}}(h)=\sqrt{h}\sum_{0\leq i+j \leq m}[\tilde{r}_{i,j}(c_1,\dots,c_{n-2})T_{i+1,j}+\tilde{s}_{i,j}(c_1,\dots,c_{n-2})T_{i,j+1}](\sqrt{h})^{i+j},
\end{align*}
where, for every $0\leq i+j \leq m$, 
\begin{align*}
\tilde{r}_{i,j}(c_1,\dots,c_{n-2}):&=-(\sqrt{2})^{i+j+1} r_{i,j}(c_1,\dots,c_{n-2}),\\
\tilde{s}_{i,j}(c_1,\dots,c_{n-2}):&=-(\sqrt{2})^{i+j+1} s_{i,j}(c_1,\dots,c_{n-2}),\\
T_{i,j}:&=\int_{0}^{2\pi}\cos^i \theta \sin^j \theta \ \mathrm{d}\theta 
=\left\{\begin{array}{l}
\dfrac{2\Gamma((i+1)/2)\Gamma((j+1)/2)}{\Gamma((i+j+2)/2)}, \text{if} \ i,j\in 2\mathbb{N}\\
0, \qquad \qquad \qquad \qquad  \qquad  \quad \text{otherwise}.\\
\end{array}\right.
\end{align*}
Taking into account that $T_{i,j}=0$ if $i\in 2\mathbb{N}+1$ or $j\in 2\mathbb{N}+1$, we obtain 
$$
\tilde{r}_{i,j}(c_1,\dots,c_{n-2})T_{i+1,j}+\tilde{s}_{i,j}(c_1,\dots,c_{n-2})T_{i,j+1}=0,
$$
for every $i,j\in\mathbb{N}$, $0\leq i+j \leq m$, such that $i+j \in 2\mathbb{N}$. Hence, the expression of the integral $I_{c_1,\dots,c_{n-2}}(h)$ becomes
\begin{align*}
I_{c_1,\dots,c_{n-2}}(h)&=\sqrt{h}\sum_{0\leq i+j \leq m,\ i+j \in 2\mathbb{N}+1}[\tilde{r}_{i,j}(c_1,\dots,c_{n-2})T_{i+1,j}+\tilde{s}_{i,j}(c_1,\dots,c_{n-2})T_{i,j+1}](\sqrt{h})^{i+j}\\
&=\left\{\begin{array}{l}
h\left( a_1 + a_3 h + \dots + a_m h^{\frac{m-1}{2}}\right),  \quad \text{if} \ m\in 2\mathbb{N}+1\\
h\left( a_1 + a_3 h + \dots + a_{m-1} h^{\frac{m-2}{2}}\right),  \ \text{if} \ m\in 2\mathbb{N}\\
\end{array}\right.,
\end{align*}
where 
\begin{align*}
a_1:&=[\tilde{r}_{1,0}(c_1,\dots,c_{n-2})+\tilde{s}_{0,1}(c_1,\dots,c_{n-2})]\pi,\\
a_3:&= \sum_{i+j=3} \tilde{r}_{i,j}(c_1,\dots,c_{n-2})T_{i+1,j}+\tilde{s}_{i,j}(c_1,\dots,c_{n-2})T_{i,j+1},\\
&\dots \\
a_{m}:&= \sum_{i+j=m} \tilde{r}_{i,j}(c_1,\dots,c_{n-2})T_{i+1,j}+\tilde{s}_{i,j}(c_1,\dots,c_{n-2})T_{i,j+1}.
\end{align*}
Consequently, $I_{c_1,\dots,c_{n-2}}(h)$ is a polynomial function, and hence its maximum number of zeros located in $(0,1/2 \rho_{c_1,\dots,c_{n-2}}^2)$, is equal to 
\begin{align}\label{estimate}
\left\{\begin{array}{l}
\dfrac{m-1}{2},  \ \text{if} \ m\in 2\mathbb{N}+1\\
\dfrac{m-2}{2},  \ \text{if} \ m\in 2\mathbb{N}.\\
\end{array}\right.
\end{align}
Note that, the above maximum values for the number of zeros of $I_{c_1,\dots,c_{n-2}}(h)$ can be achieved for certain $Q_1$ and $Q_2$ of the type \eqref{poly} such that the constants $\tilde{r}_{i,j}(c_1,\dots,c_{n-2})$, and $\tilde{s}_{i,j}(c_1,\dots,c_{n-2})$ are carefully chosen in accordance with $a_i$'s.

Now one can give the main results of this section, which provide a sharp upper bound estimate for the number of leafwise bifurcating limit cycles of the perturbed systems \eqref{hyperpert1} and \eqref{hyperpert2}, if the perturbation vector field is generated by analytic functions of the type \eqref{poly}. 

Let us state now the first result, concerning the perturbed system \eqref{hyperpert1}, which keeps dynamically invariant an a-priori fixed regular symplectic leaf, $\Sigma_{c_1 ,\dots, c_{n-2}}\cap\Omega^{\prime}_{\varepsilon_3,\dots,\varepsilon_n}$, of the Poisson manifold $(\Omega^{\prime}_{\varepsilon_3,\dots,\varepsilon_n},\{\cdot,\cdot\}_{\nu;C_1,\dots,C_{n-2}})$.

\begin{theorem}\label{mainThmExampleConcret1}
If $I_{c_1,\dots,c_{n-2}}(h)\not\equiv 0$ on $(0,\dfrac{1}{2}\rho_{c_1,\dots,c_{n-2}}^2)$, then for each choice of $\varepsilon_3,\dots,\varepsilon_n\in\{-1,1\}$, the total number (counting also the multiplicities) of limit cycles of the perturbed system \eqref{hyperpert1} with $P_1$ and $P_2$ of the type \eqref{poly}, located on the symplectic leaf $\Sigma_{c_1 ,\dots, c_{n-2}}\cap\Omega^{\prime}_{\varepsilon_3,\dots,\varepsilon_n}$, and bifurcating from the family $\bigcup_{h\in (0,1/2\rho_{c_1,\dots,c_{n-2}}^2)}\gamma_h^{\varepsilon_3,\dots,\varepsilon_n}$ \eqref{circles} of periodic orbits of the completely integrable \eqref{hyperelli}, is smaller than or equal to 
\begin{align*}
\left\{\begin{array}{l}
\dfrac{m-1}{2},  \ \text{if} \ m\in 2\mathbb{N}+1\\
\dfrac{m-2}{2},  \ \text{if} \ m\in 2\mathbb{N}.\\
\end{array}\right.
\end{align*}
\end{theorem}
\begin{proof}
The conclusion follows directly from Theorem \eqref{mainThmExample1} and the relation \eqref{estimate}.
\end{proof}

Let us state now the second result, concerning the perturbed system \eqref{hyperpert2}, which keeps dynamically invariant the regular part of the symplectic foliation of the Poisson manifold $(\Omega^{\prime}_{\varepsilon_3,\dots,\varepsilon_n},\{\cdot,\cdot\}_{\nu;C_1,\dots,C_{n-2}})$. Before stating the result, recall that for every $c_1,\dots,c_{n-2}>0$, $J_{c_1,\dots,c_{n-2}}(h)=I_{c_1,\dots,c_{n-2}}(h)$, for any $h\in(0,\dfrac{1}{2}\rho_{c_1,\dots,c_{n-2}}^2)$.

\begin{theorem}\label{mainThmExampleConcret2}
Suppose there exists $c_1^0,\dots,c_{n-2}^0>0$ such that the map $J_{c_1^0,\dots,c_{n-2}^0}(h)\not\equiv 0$ on $(0,\dfrac{1}{2}\rho_{c_1^0,\dots,c_{n-2}^0}^2)$. Then for every $\varepsilon_3,\dots,\varepsilon_n\in\{-1,1\}$, and respectively for every $c_1,\dots,c_{n-2}>0$ such that $J_{c_1\dots,c_{n-2}}(h)\not\equiv 0$ on $(0,\dfrac{1}{2}\rho_{c_1,\dots,c_{n-2}}^2)$, the total number (counting also the multiplicities) of limit cycles of the perturbed system \eqref{hyperpert2} with $Q_1$ and $Q_2$ of the type \eqref{poly}, located on the symplectic leaf $\Sigma_{c_1 ,\dots, c_{n-2}}\cap\Omega^{\prime}_{\varepsilon_3,\dots,\varepsilon_n}$, and bifurcating from the family $\bigcup_{h\in (0,1/2\rho_{c_1,\dots,c_{n-2}}^2)}\gamma_h^{\varepsilon_3,\dots,\varepsilon_n}$ \eqref{circles} of periodic orbits of the completely integrable \eqref{hyperelli}, is smaller than or equal to 
\begin{align*}
\left\{\begin{array}{l}
\dfrac{m-1}{2},  \ \text{if} \ m\in 2\mathbb{N}+1\\
\dfrac{m-2}{2},  \ \text{if} \ m\in 2\mathbb{N}.\\
\end{array}\right.
\end{align*}
\end{theorem}
\begin{proof}
The conclusion follows directly from Theorem \eqref{mainThmExample2} and the relation \eqref{estimate}.
\end{proof}


\bigskip
\bigskip

\noindent {\sc R.M. Tudoran}\\
West University of Timi\c soara\\
Faculty of Mathematics and Computer Science\\
Department of Mathematics\\
Blvd. Vasile P\^arvan, No. 4\\
300223 - Timi\c soara, Rom\^ania.\\
E-mail: {\sf tudoran@math.uvt.ro}\\
\medskip

\noindent {\sc A. G\^irban}\\
"Politehnica" University of Timi\c soara,\\
Department of Mathematics,\\
Pia\c ta Victoriei, No. 2,\\
300006 - Timi\c soara, Rom\^ania.\\
E-mail: {\sf anania.girban@gmail.com}\\


\begin{thebibliography}{99}

\bibitem{arnoldd} {\footnotesize \textsc{V.I. Arnold, S.M. Gusein-Zade and A.N. Varchenko}, \textit{Singularities of Differentiable Maps}, Volume II, Monodromy and Asymptotics of Integrals, Monographs in Mathematics, Vol. 83, Birkh\"auser 1988. }

\bibitem{bates} {\footnotesize \textsc{L. Bates and R. Cushman}, Complete integrability beyond Liouville-Arnold, \textit{Rep. Math. Phys.}, 56(1)(2005), 77--91. }

\bibitem{bermejo1} {\footnotesize \textsc{I.A. Garc\'ia and B. Hern\'andez-Bermejo}, Perturbed Euler top and bifurcations of limit cycles on invariant Casimir surfaces, \textit{Physica D}, 239(2010), 1665--1669. }

\bibitem{bermejo2} {\footnotesize \textsc{I.A. Garc\'ia and B. Hern\'andez-Bermejo}, Periodic orbits in analytically perturbed Poisson systems, \textit{Physica D}, 276(2014), 1--6. }

\bibitem{cli} {\footnotesize \textsc{C. Christopher and C. Li}, \textit{Limit cycles of differential equations}, Advanced Courses on Mathematics CRM Barcelona, Birkh\"auser Verlag, 2007. }

\bibitem{g1} {\footnotesize \textsc{A. Ay, M. G\"urses and K. Zheltukhin}, Hamiltonian equations in $\R^3$, \textit{J. Math. Phys.}, 44(2003), 5688--5705. }

\bibitem{berme1} {\footnotesize \textsc{B. Hern\'andez-Bermejo}, Generalization of solutions of the Jacobi PDEs associated to time reparametrizations of Poisson systems, \textit{J. Math. Anal. Appl.}, 344(2008), 655--666. }


\bibitem{g2} {\footnotesize \textsc{M. G\"urses, G.S. Guseinov and K. Zheltukhin}, Dynamical systems and Poisson structures, \textit{J. Math. Phys.}, 50(2009), 112703. }

\bibitem{nambu1} {\footnotesize \textsc{Y. Nambu}, Generalized Hamiltonian mechanics, \textit{Phys. Rev. D},7(8)(1973), 2405--2412. }

\bibitem{poincare} {\footnotesize \textsc{H. Poincar\'e}, Sur le probl\`eme des trois corps et les \'equations de la dynamique, \textit{Acta Math.}, XIII(1890), 1--270. }

\bibitem{pontryagin} {\footnotesize \textsc{L. Pontryagin}, On dynamical systems close to Hamiltonian ones, \textit{Zh. Exp. and Theor. Phys.}, 4(1934), 234--238. }

\bibitem{ratiurazvan} {\footnotesize \textsc{T.S. Ratiu, R.M. Tudoran, L. Sbano, E. Sousa Dias and G. Terra}, \textit{Geometric Mechanics and Symmetry: the Peyresq Lectures; Chapter II: A Crash Course in Geometric Mechanics}, pp. 23--156, London Mathematical Society Lecture Notes Series, vol. 306, Cambridge University Press 2005. }

\bibitem{nambu2} {\footnotesize \textsc{L. Takhtajan}, On foundation of the generalized Nambu mechanics, \textit{Comm. Math. Phys.},160(1994), 295--315. }

\bibitem{tudoran} {\footnotesize \textsc{R.M. Tudoran}, A normal form of completely integrable systems, \textit{J. Geom. Phys.}, 62(5)(2012), 1167--1174. }

\bibitem{tudoran1} {\footnotesize \textsc{R.M. Tudoran}, Affine distributions on Riemannian manifolds with applications to dissipative dynamics, \textit{J. Geom. Phys.}, 92(2015), 55--68. }

\end{thebibliography}
\end{document}